\documentclass[final]{MyElsarticle}

\usepackage{framed,multirow}

\usepackage{geometry}
 \global\bibsep=0pt
 \input{fleqn.clo}

\usepackage{graphicx}
\usepackage{amssymb}
\usepackage{latexsym}
\usepackage[fleqn]{amsmath}
\usepackage{mathtools}
\usepackage[overload]{empheq}

\usepackage{url}
\usepackage[dvipsnames]{xcolor}
\usepackage{nicefrac}
\usepackage{xfrac}
\usepackage[normalem]{ulem}

\usepackage{lineno}

\biboptions{square,sort&compress}
\journal{Journal of Computational Physics}

\definecolor{MSBlue}{rgb}{.204,.353,.541}
\definecolor{MSLightBlue}{rgb}{.31,.506,.741}

\usepackage[T1]{fontenc}
\usepackage[sfscaled=true,cal=cmcal,charter]{mathdesign}

    \usepackage{caption}

\colorlet{HREFCOLOR}{RoyalBlue}
\usepackage{hyperref}
\hypersetup{%
  breaklinks,
  colorlinks=true,
  linkcolor=HREFCOLOR,
  anchorcolor=HREFCOLOR,
  citecolor=HREFCOLOR,
  filecolor=HREFCOLOR,
  menucolor=HREFCOLOR,
  urlcolor=HREFCOLOR,
  bookmarksnumbered,
  pdfencoding=auto,
  pdfstartview=FitH,
  pdfpagemode=UseNone,
  pdfnewwindow=true,
  bookmarksopen=false,
  baseurl={}
}

\def\cf{\emph{cf.}}
\def\ie{\emph{i.e.}}
\def\eg{\emph{e.g.}}
\def\etc{\emph{etc.}}
\def\resp{{resp.}}

\def\ZZ{\mathbb{Z}}
\def\RR{\mathbb{R}}

\def\CC{\mathbb{C}}

\def\bi{\mathbf{i}}

\def\bk{\mathbf{k}}

\def\bx{\mathbf{x}}

\def\bz{\mathbf{z}}

\DeclareMathOperator{\re}{Re}         
\DeclareMathOperator{\im}{Im}         
\DeclareMathOperator{\sgn}{sgn}       
\DeclareMathOperator{\spanset}{Span}  
\newcommand{\overbar}[1]{\overline{#1}}
\newcommand*{\conj}[1]{\overbar{#1}}  

\newcommand{\idx}[1]{{#1}}

\newcommand{\half}{{\sfrac{1}{2}}}

\newcommand{\iu}{{{\iota\mkern1mu}}} 
\newcommand{\di}{{\alpha}} 
\newcommand{\dx}{{h\mkern1mu}} 
\newcommand{\cK}{{\mathcal{K}}} 

\newcommand{\figref}[1]{\hyperref[#1]{\upshape Fig.~\ref{#1}}}
\newcommand{\defref}[1]{\hyperref[#1]{\upshape Def.~\ref{#1}}}
\renewcommand{\eqref}[1]{\hyperref[#1]{\upshape (\ref{#1})}}
\newcommand{\teqref}[1]{\hyperref[#1]{\upshape Eq.~(\ref{#1})}}
\newcommand{\secref}[1]{\hyperref[#1]{\upshape Sec.~\ref{#1}}}
\newcommand{\appref}[1]{\hyperref[#1]{\upshape App.~\ref{#1}}}
\newcommand{\thmref}[1]{\hyperref[#1]{\upshape Thm.~\ref{#1}}}
\newcommand{\propref}[1]{\hyperref[#1]{\upshape Prop.~\ref{#1}}}
\newcommand{\lemref}[1]{\hyperref[#1]{\upshape Lemma~\ref{#1}}}
\newcommand{\algref}[1]{\hyperref[#1]{\upshape Alg.~\ref{#1}}}

\newtheorem{theorem}{Theorem}[section]
\newtheorem{lemma}{Lemma}[section]
\newtheorem{proposition}{Proposition}[section]

\newtheorem{definition}{Definition}[section]

\newdefinition{remark}{Remark}

\newproof{proof}{Proof}

\newcommand{\revise}[1]{#1}

\renewcommand{\sout}[1]{\!}

\begin{document}
\begin{frontmatter}


\title{A Reflectionless Discrete Perfectly Matched Layer}



\author{Albert {Chern}
}
  \ead{chern@math.tu-berlin.de}
\address{Institute of Mathematics, Technical University of Berlin, 10623 Berlin, Germany}


\begin{abstract}
Perfectly Matched Layer (PML) is a widely adopted non-reflecting boundary treatment for wave simulations.  Reducing numerical reflections from a discretized PML has been a long lasting challenge.  This paper presents a new discrete PML for the multi-dimensional scalar wave equation which produces no numerical reflection at all.  The reflectionless discrete PML is discovered through a straightforward derivation using Discrete Complex Analysis.  The resulting PML takes an easily-implementable finite difference form with compact stencil.  In practice, the discrete waves are damped exponentially in the PML, and the error due to domain truncation is maintained at machine zero by a moderately thick PML.  The numerical stability of the proposed PML is also demonstrated.

\end{abstract}


\end{frontmatter}


\section{Introduction}
\label{sec:Introduction}

Perfectly Matched Layer (PML)
 \cite{Berenger:1994:PML,Gredney:2005:PML,Johnson:2008:NPM}
 is a state-of-the-art numerical technique for absorbing waves at the boundary of a computation domain, highly demanded by simulators of wave propagations in free space.   More precisely, PML is an artificial layer attached to the boundary where the wave equation is modified into a set of so-called PML equations. Requiring simulations only \emph{local} in space and time, the layer effectively damps the waves while creating no reflection at the transitional interface (\figref{fig:WaterShade}).  Although the reflectionless property is well established in the continuous theory, to the best knowledge all existing discretizations of PML equations still produce numerical reflections at the interface due to discretization error (see \eg~\cite[Sec.~7.1]{Johnson:2008:NPM}, \cite[Intro.]{Komatitsch:2007:UCP}, \cite[Sec.~1]{Hagstrom:2014:DAB}, and a review in \secref{sec:RelatedWork}).

This paper presents a new approach to deriving discrete PML equations.  Instead of seeking a high-order discretization of the continuous PML equations, we take the discrete wave equation and find its associated PML equations mimicking the continuous theory but solely in the discrete setting.  The derivation particularly uses \emph{Discrete Complex Analysis}
 \cite{Duffin:1956:BPD,Bobenko:2005:LNT,Lovasz:2004:DAF,Bobenko:2016:DCA}.  
 The resulting discrete PML for the first time ``perfectly matches'' the discrete wave equation.  In particular, it produces \emph{no} numerical reflection at the interface or at any transition where the PML damping parameter varies. This reflectionless property holds for all wavelengths, even for those at the scale of the grid size.

\begin{figure}
	\centering
	\includegraphics[width=0.5\textwidth]{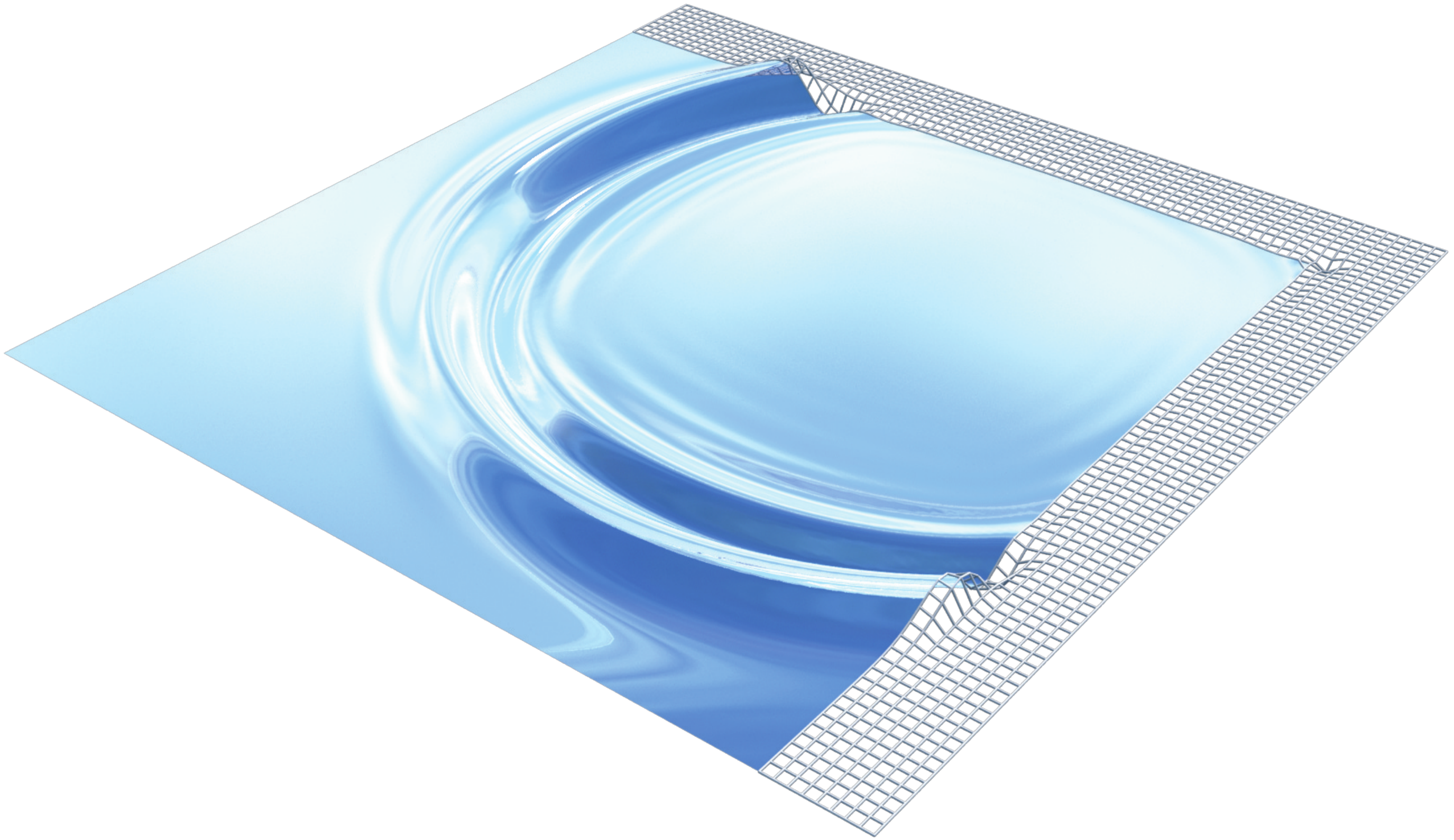}
	\caption{\label{fig:WaterShade}PML (shown as grids) is a layer attached to the boundary absorbing the waves exiting the computation domain.}
\end{figure}

Although the proposed discrete PML produces no transitional reflection, its finite absorption rate does not prevent a small amount of waves to travel all the way through the layer.  In practice, when the layer has to be truncated into a finite thickness, these remaining waves eventually return to the physical domain (\figref{fig:ResidualWater}) either through a periodic boundary, or due to a reflecting boundary at the end of the layer.  These returning waves are referred as the \emph{residual waves} (also known as the \emph{round-trip reflections}).  Note particularly that a residual wave is fundamentally different from a numerical (transitional) reflection: The former is already present in the continuous theory, while the latter is completely avoidable as shown in this paper.  

\begin{figure}[b]
	\centering
	\includegraphics[width=1\textwidth]{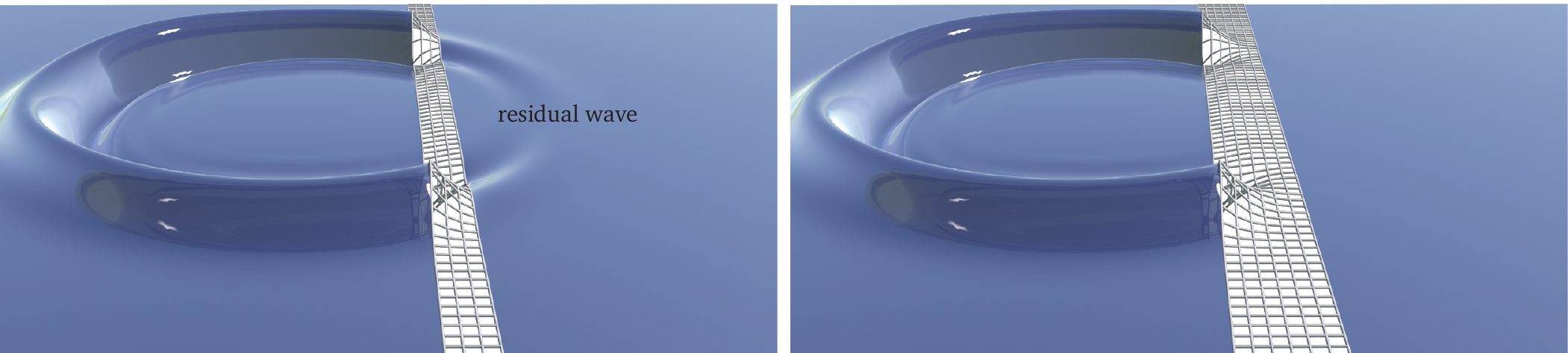}
	\caption{\label{fig:ResidualWater}Waves in the PML (shown as grids) are damped exponentially.  Residual waves, which re-enter the physical domain through a periodic boundary (left), are essentially removed by a sufficiently thick PML (right). There is no numerical reflection at the PML interface in the simulation.}
\end{figure}

In the absence of numerical reflection, the maximal residual wave is scaled down exponentially as the PML is thicken linearly (\figref{fig:ResidualWater}, right) without any grid refinement. 
As a result, the residual waves can also be suppressed below machine zero with a moderately thick PML 
(\secref{sec:RoundtripReflections}).  
In addition, numerical evidence shows that the proposed discrete PML is \emph{stable} (\secref{sec:Stability}).

Before giving a more thorough survey on the related work in \secref{sec:RelatedWork}, I present an overview of the results.

\subsection{Main Results}
\label{sec:MainResults}
\revise{
\subsubsection{Setup}
}
Consider the PML problem for the \((d+1)\)-dimensional scalar wave equation:
\begin{equation}
\label{eq:WaveEquation}
\frac{\partial^2 u}{\partial t^2}(t,\bx) = \sum_{\di=1}^d\frac{\partial^2 u}{\partial x_\di^2} (t,\bx),\quad \bx = (x_1,\ldots,x_d)\in\RR^d.
\end{equation}
Adopting the 2\(^{\textrm{nd}}\)-order central difference yields the semi-discrete wave equation on a regular grid with grid size \(\dx>0\):
\begin{align}
\label{eq:DiscreteWaveEquation}
\frac{\partial^2 U}{\partial t^2}(t,\bi) = \sum_{\di=1}^d\frac{1}{\dx^2}\left(-2U(t,\bi)+(\tau_\di^{-1}U)(t,\bi)+(\tau_\di U)(t,\bi)\right)
\end{align}
for \(\bi=(\idx{i}_1,\ldots,\idx{i}_d)\in \ZZ^d\), where \(U(t,\bi)\) stands for an approximation of \(u(t,\bx)\vert_{\bx = \dx\bi}\), and \(\tau_\di^{\pm 1}\) the shift operator
\[(\tau_\di^{\pm 1} U)(t,\idx{i}_1,\ldots,\idx{i}_\di,\ldots,\idx{i}_d)\coloneqq U(t,\idx{i}_1,\ldots,\idx{i}_\di\pm 1,\ldots \idx{i}_d).\]
With the main focus being the reflectionless property of the PML, we consider the multi-half-space problem; that is, we leave \(\{\idx{i}_1<0\}\cap\cdots\cap\{\idx{i}_d<0\}\) as the physical domain, and let \(\{\idx{i}_1\geq 0\}\cup\cdots\cup\{\idx{i}_d\geq 0\}\) be the domain for PML. 

\begin{figure}
	\centering
	\includegraphics[width=0.45\textwidth]{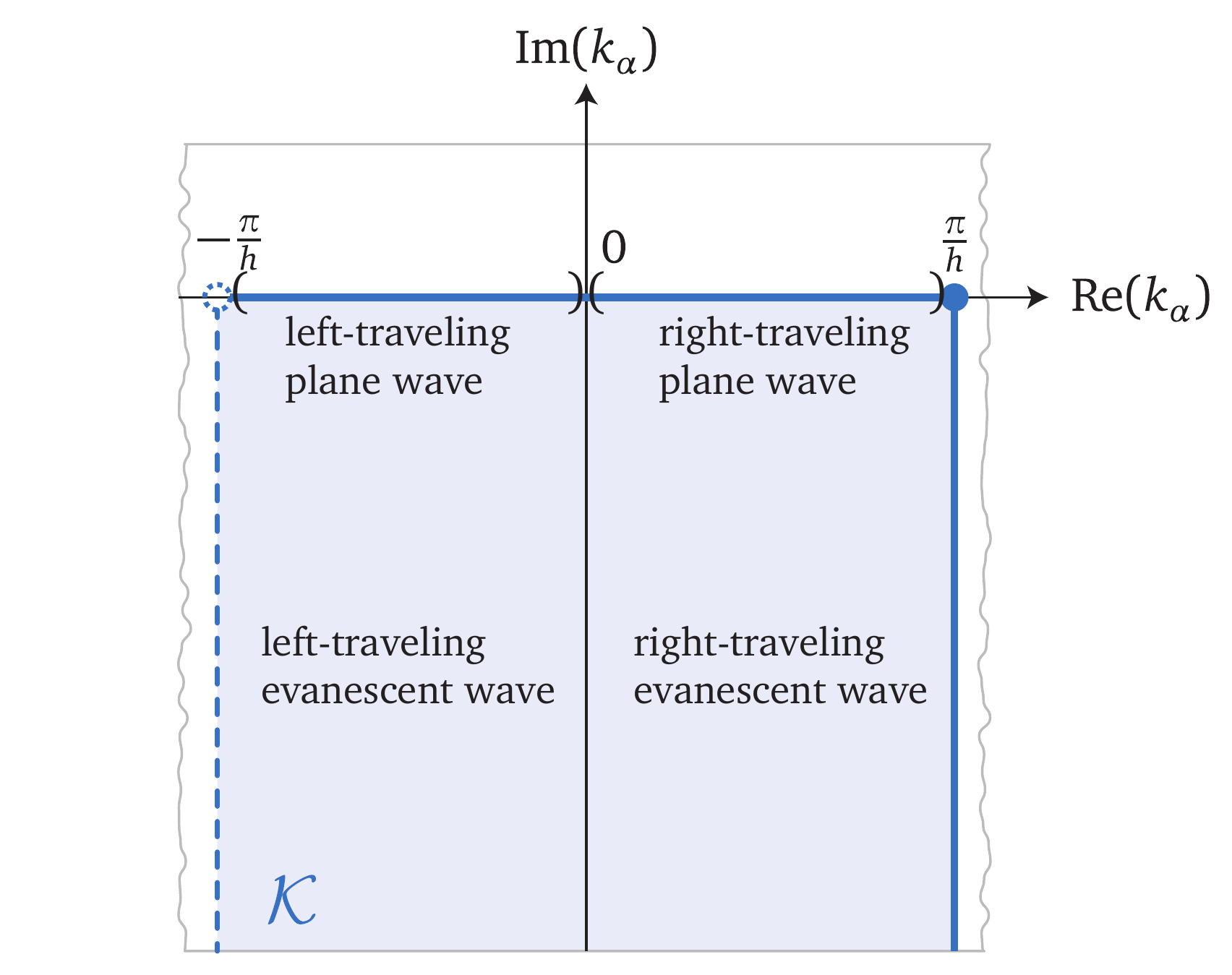}
	\caption{\label{fig:KDomain}\revise{The feasible region \(\cK\subset\CC\) for each wave number component \(k_\di\) of a traveling wave solution \(U(t,\bi) = e^{-\iu\omega t}e^{\iu\sum_{\di=1}^d k_\di\dx\idx{i}_\di}\) of \eqref{eq:DiscreteWaveEquation} permitted in the physical domain \(\bigcap_{\di=1}^d\{\idx{i}_\di < 0\}\).  The annotations about the left- and right-traveling directions are with respect to the case \(\omega>0\).}}
\end{figure}
\revise{Here are a few preliminary facts. In this physical domain the solution consists of elementary waves \(U(t,\bi) = e^{-\iu\omega t}e^{\iu\sum_{\di=1}^d k_\di\dx\idx{i}_\di}\), where \(\omega\in\RR\) and \(k_\di\in\CC\) satisfy the \emph{discrete dispersion relation} 
\(\omega^2\dx^2 = \sum_{\di=1}^d 4\sin^2\left(\nicefrac{k_\di\dx}{2}\right)\).
The feasible values of wave numbers \(k_\di\) are further restricted to
\begin{equation}
    \label{eq:KDomain}
    \cK \coloneqq \left\{k_\di\in\CC\,\vert\,-\tfrac{\pi}{\dx}<\re(k_\di)\leq \tfrac{\pi}{\dx}, \im(k_\di)<0\right\}\subset\CC
\end{equation}
using the periodicity of \(e^{\iu k_\di\dx\idx{i}_\di}\) and the boundedness assumption of \(U(t,\bi)\) on the physical domain \(\bigcap_{\di=1}^d\{\idx{i}_\di < 0\}\).  The wave travels in the positive (resp.\ negative) \(\di\)-direction if \(\sgn(\omega)\re(k_\di)\in(0,\nicefrac{\pi}{\dx})\) (resp.\ \(\in(-\nicefrac{\pi}{\dx},0)\)).  We also call the positive direction \emph{right} and the negative direction \emph{left}.  The wave that has \(\im(k_\di)=0\) for all \(\di\) is called a \emph{plane wave}. The wave that has \(\im(k_\di)<0\) for some \(\di\) is called an \emph{evanescent wave}. See \figref{fig:KDomain}.
}

\subsubsection{The Discrete PML Equations}
\label{sec:IntroDiscretePMLEquations}

In the continuous theory, the PML equations are derived by a change of coordinates in the analytic continuation of the wave equation.  This technique is known as \emph{complex coordinate stretching} \cite{Chew:1994:A3D,Johnson:2008:NPM}, which we will review in \secref{sec:ComplexCoordinateStretching}.  \emph{Discrete Complex Analysis} (\secref{sec:DiscreteComplexAnalysis}) is a study of complex-valued functions supported on a lattice analogous to the theory of holomorphic functions in complex analysis.  Using Discrete Complex Analysis, we present a discrete version of complex coordinate stretching (\secref{sec:DerivationOfDiscretePML}), leading to the discrete PML for the discrete wave equation \eqref{eq:DiscreteWaveEquation}:
\begin{subequations}
\label{eq:DiscretePML}
\begin{align}[left=\empheqlbrace]
	\label{eq:DiscretePMLa}
		&\frac{\partial^2 U}{\partial t^2} = \sum_{\di=1}^d \frac{1}{\dx^2}\left(-2U + \tau_\di^{-1}U + \tau_\di U\right) + \sum_{\di=1}^d\frac{1}{\dx}\left((\sigma_\di)(\tau_\di \Psi_\di) - (\tau_\di^{-1}\sigma_\di)(\tau_\di^{-1}\Phi_\di)\right)\\
	\label{eq:DiscretePMLb}
		&\frac{\partial\Phi_\di}{\partial t} = -\frac{1}{2}\left((\tau_\di^{-1}\sigma_\di)(\tau_\di^{-1}\Phi_\di) + \sigma_\di\Phi_\di\right) - \frac{1}{2\dx}\left(\tau_\di U - \tau_\di^{-1}U\right)\\
	\label{eq:DiscretePMLc}
		&\frac{\partial\Psi_\di}{\partial t} = -\frac{1}{2}\left((\tau_\di^{-1}\sigma_\di)(\Psi_\di) + (\sigma_\di)(\tau_\di\Psi_\di)\right) - \frac{1}{2\dx}\left(\tau_\di U - \tau_\di^{-1}U\right).
\end{align}
\end{subequations}
Eq.~\eqref{eq:DiscretePML} is a time evolution for \(U\) and auxiliary variables \(\Phi_\di,\Psi_\di\) on \(\ZZ^d\). The parameters \(\sigma_\di(\bi) = \sigma_\di(\idx{i}_\di)\geq 0\) are the \emph{damping coefficients}, with \(\sigma_\di(\idx{i}_\di) = 0\) for \(\idx{i}_\di<0\). In the physical domain, where \(\sigma_\di=0\) for all \(\di=1,\ldots,d\), Eq.~\eqref{eq:DiscretePMLa} is reduced to Eq.~\eqref{eq:DiscreteWaveEquation}. The auxiliary variables \(\Phi_\di\) (\resp~\(\Psi_\di\)) and their evolution equations \eqref{eq:DiscretePMLb} (\resp~\eqref{eq:DiscretePMLc}) only need to be supported on the half-spaces \(\{\bi\in \ZZ^d\,\vert\,\idx{i}_\di\geq 0\}\) (\resp~on \(\{\bi\in \ZZ^d\,\vert\,\idx{i}_\di\geq 1\}\)), \ie~in the PML.

The discrete PML equations \eqref{eq:DiscretePML} will be derived in \secref{sec:DerivationOfDiscretePML}.  Alongside the derivation is the analysis of discrete  waves on a discrete complex domain, establishing the following result as a corollary of the theory.

\begin{theorem}[Reflectionless Property of the Discrete PML]
	\label{thm:Reflectionless}
	Let \(U(t,\bi) = e^{-\iu\omega t}e^{\iu\sum_{\di=1}^d k_\di\dx\idx{i}_\di}\), 
\revise{\(k_\di\in\cK\),}
 \(\omega\in\RR\), \(\omega\neq 0\), be a discrete traveling wave that satisfies \textup{Eq.}~\eqref{eq:DiscreteWaveEquation} in the physical domain \(\bigcap_{\di=1}^d\{\idx{i}_\di<0 \}\).  Then \(U(t,\bi)\) can be uniquely extended to a solution of \textup{Eq.}~\eqref{eq:DiscretePML} for all \(\bi\in\ZZ^d\):
	\begin{align}
		\label{eq:IntroSolution}
		U(t,\bi) = \left(\prod_{\di=1}^dr_\di(\idx{i}_\di)\right)e^{-\iu\omega t}e^{\iu\sum_{\di=1}^d k_\di\dx\idx{i}_\di},\quad
		r_\di(\idx{i}) = 
		\begin{cases}
			1,&\idx{i}\leq 0\\
			\prod_{\idx{j}=0}^{\idx{i}-1}\rho(\sigma_\di(\idx{j}),k_\di,\omega),&\idx{i}>0,
		\end{cases}
	\end{align}
	where the geometric decay rate
	\begin{align}
		\label{eq:IntroDecayRate}
		\rho(s,k_\di,\omega) = \frac{2 + \iu\frac{s}{\omega}\left(1-e^{-\iu k_\di\dx}\right)}{2 + \iu\frac{s}{\omega}\left(1-e^{\iu k_\di\dx}\right)}
	\end{align}
    \revise{
    satisfies for \(s>0\), \(\omega\in\RR\), \(k_\di\in\cK\) that
    \begin{subequations}
    \begin{align}
        \label{eq:IntroDecayBound}
        |\rho(s,k_\di,\omega)|<1
    \end{align}
    and
    \begin{align}
        \label{eq:IntroReflectedDecayBound}
        \frac{|\rho(s,k_\di,\omega)|}{\left|\rho\left(s,-\conj{ k_\di},\omega\right)\right|}<1
    \end{align}
    \end{subequations}
     whenever
    \(\sgn(\omega)\re(k_\di)\in(0,\nicefrac{\pi}{\dx})\).
    }
\end{theorem}

\thmref{thm:Reflectionless} guarantees that every traveling \revise{(plane or evanescent)} wave incident to the discrete PML is perfectly transmitted, and \revise{by \eqref{eq:IntroDecayBound}} the amplitude of the wave is successively damped along  \revise{the positive (\(\sgn(\omega)\re(k_\di)\in(0,\nicefrac{\pi}{\dx})\))} direction in the layer. 
This property is true \emph{independent} of the choice of the spatially varying damping coefficients \(\sigma_\di\).  It is worth noting that one of the major numerical studies in previous work is on the profile of \(\sigma_\di\) on which the transitional reflections depend. Multiple studies have drawn a conclusion that \(\sigma_\di\) has to be ``turned on gradually'' in order to allow a less-reflective adiabatic process. In the case of Eq.~\eqref{eq:DiscretePML} such a design of \(\sigma_\di\) becomes irrelevent concerning the transitional reflections. 

\revise{Should there be any reflecting boundary at \(\idx{i}_\di = M\) that truncates the layer, a right-traveling wave would give rise to a reflected left-traveling wave with wave number \(-\conj{k_\di} = -\re(k_\di)+\iu\im(k_\di)\) and a matching amplitude at \(\idx{i}_\di = M\).  As this left-traveling wave reaches the physical domain, the amplitude has been suppressed by the multiplicatively cumulative amount of \(\prod_{\idx{i}=0}^{M-1}|\rho(s,k_\di,\omega)||\rho(s,-\conj{k_\di},\omega)|^{-1}\).   Eq.~\eqref{eq:IntroReflectedDecayBound} of \thmref{thm:Reflectionless} assures the effectiveness of such a truncated PML backed by a reflecting boundary.
In this paper, we truncate the PML with a periodic boundary condition, in which case \eqref{eq:IntroReflectedDecayBound} is never invoked.}

\subsubsection{Numerical Examples}
\label{sec:NumericalExamples}

\begin{figure}[t]
	\centering
	\includegraphics[width=\textwidth]{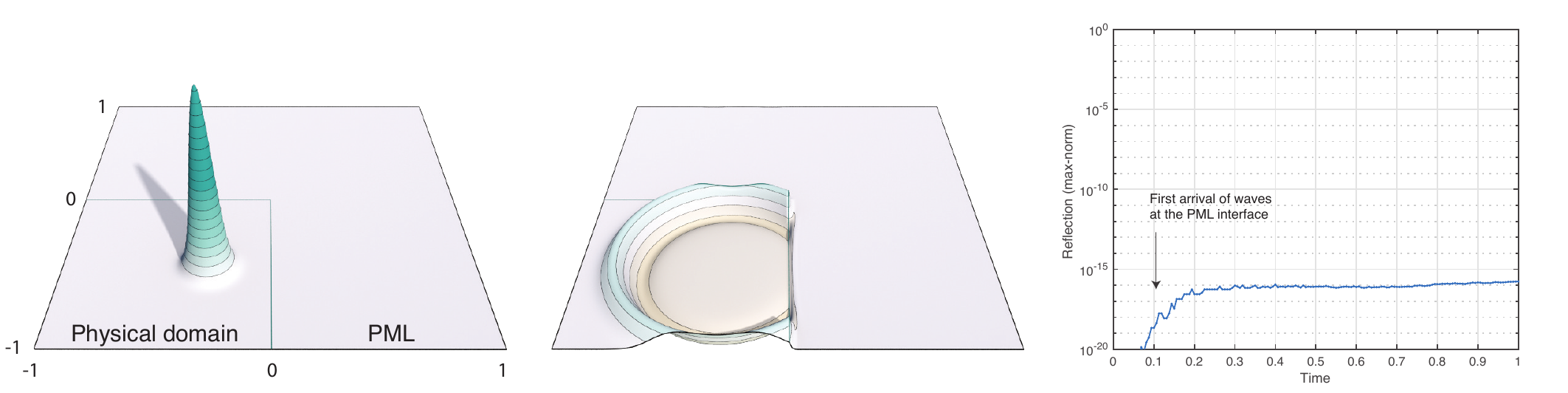}
	\caption{\label{fig:HalfSpaceGauss}A simulation of Eq.~\eqref{eq:DiscretePML} on \([-1,1]^2\) with \(\dx = \nicefrac{1}{80}\), periodic boundary, and constant damping coefficient \(\sigma_\di(\idx{i}_\di) = \nicefrac{2}{\dx}\) for \(\idx{i}_\di\geq 0\). The initial condition is a bump with peak value \(2\), and \((\nicefrac{\partial U}{\partial t})(0,\bi) = 0\). From left to right shows the initial condition at \(t=0\), the solution at \(t=0.5\), and the reflection error \(\max|U - U_{\rm ref}|\) over the physical domain (computed with the double precision). Here the reference ``free-space'' solution \(U_{\rm ref}\) is the solution to Eq.~\eqref{eq:DiscreteWaveEquation} on \([-5,5]^2\).}
\end{figure}

\begin{figure}
	\centering
	\includegraphics[width=\textwidth]{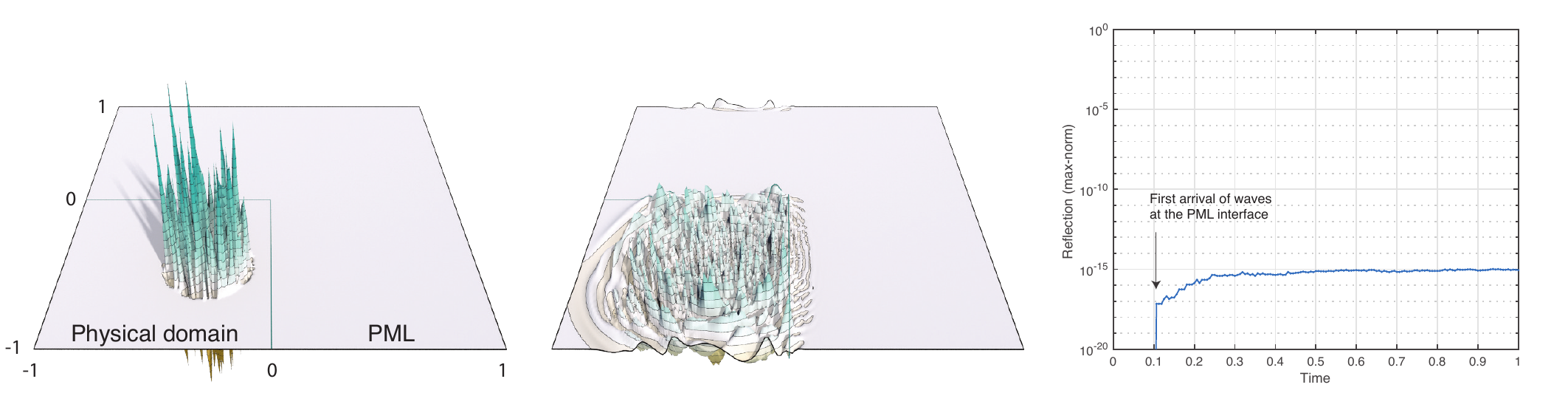}
	\caption{\label{fig:HalfSpaceRand}A simulation of Eq.~\eqref{eq:DiscretePML} with the same setup as \figref{fig:HalfSpaceGauss}, but with random damping coefficients \(\sigma_\di(\idx{i}_\di)\overset{\rm i.i.d.}{\sim}\operatorname{Unif}([0,\nicefrac{2}{\dx}])\), and a noisy initial condition (Gaussian white noise in a disk with variance \(\nicefrac{1}{4}\)). From left to right shows the initial condition at \(t=0\), the solution at \(t=0.5\), and the maximal reflection measured in the same way as in \figref{fig:HalfSpaceGauss}. Note that the numerical reflection is at machine zero despite that the waves are barely resolved on the grid.}
\end{figure}
As numerical validations, \figref{fig:HalfSpaceGauss} and \ref{fig:HalfSpaceRand} show simulations of Eq.~\eqref{eq:DiscretePML} on a 2D domain. In these examples, the multi-half-space is truncated into a finite domain \([-1,1]^2\ni\dx\bi\) with periodic boundary, so that the physical domain \([-1,0]^2\) is ``surrounded'' by the PML.  The \(4^\text{th}\)-order Runge-Kutta method (RK4) is adopted for time integration.  The numerical reflection is measured directly by the max error \(\|U - U_{\rm ref}\|_{\rm max}\) in the physical domain against a reference solution \(U_{\rm ref}\) solving \eqref{eq:DiscreteWaveEquation} with RK4 on a larger domain \([-\ell,\ell]^2, \ell\gg 1\).  \figref{fig:HalfSpaceGauss} shows a benchmark where the initial condition is a Gauss ``bump'' function, and the PML has a constant damping coefficient \(\sigma_\di(\idx{i}_\di) = \nicefrac{2}{\dx}\) for \(\idx{i}_\di\geq 0\).   \figref{fig:HalfSpaceRand} demonstrates a ``more challenging'' setup where the initial condition is a Gaussian white noise in a compact region, and each of the damping coefficients \((\sigma_\di(0),\sigma_\di(1),\ldots)\), \(\di=1,2\), are \emph{independently}, \emph{randomly} chosen.  In both numerical examples, the measured reflection is at machine zero when the waves enter the PML.  In particular, \figref{fig:HalfSpaceRand} validates that the reflectionless property holds regardless of the profile of \(\sigma_\di\) or the incident wave vector.

\subsubsection{Decay Rate and Residual Waves}
\label{sec:RoundtripReflections}

When the PML is truncated into a finite thickness such as described in \secref{sec:NumericalExamples}, a residual wave will re-enter the physical domain.  The magnitude of these remaining waves can be explicitly calculated by multiplicatively cumulating the one-grid decay rate \(\rho(\sigma_\di,k_\di,\omega)\) (Eq.~\eqref{eq:IntroDecayRate}). Since \(|\rho(\sigma_\di,k_\di,\omega)|<1\) for all \(\sigma_\di>0\) and traveling waves, presumably any choice of \(\sigma_\di\) results in an exponential decay in the residual waves as the PML is thicken linearly.
In practice, however, as one faces the decision of choosing \(\sigma_\di\) for an economic layer thickness, one must understand the basic
 relation between \(\rho\) and its arguments, which I now briefly explain. Practical advice about the choice of \(\sigma_\di\) will also be included. 
For simplicity, smooth asymptotics and some heuristic arguments may be applied.

The frequency variables \(k_\di\), \(\omega\) are linked by the \emph{discrete dispersion relation} \((\omega\dx)^2 = \sum_{\di=1}^d4\sin^2(\nicefrac{k_\di\dx}{2})\), which is the condition that the plane wave \(U(t,\bi) = \smash{e^{-\iu\omega t}e^{\iu\sum_{\di=1}^d k_\di\dx\idx{i}_\di}}\) solves Eq.~\eqref{eq:DiscreteWaveEquation}.  Using the dispersion relation, it is straightforward to see that in the smooth asymptotics \(k_\di\dx\sim O(\dx)\), we have \(\omega\dx\sim O(\dx)\), \((1-e^{\pm\iu k_\di\dx})\sim\mp\iu k_\di\dx\), and
\begin{align}
	\label{eq:RhoAsymptotics}
	\rho(\sigma_\di,k_\di,\omega)\sim\frac{2 - \sigma_\di\dx\frac{k_\di}{\omega}}{2+\sigma_\di\dx\frac{k_\di}{\omega}}.
\end{align}
Therefore, asymptotically the decay rate \(\rho\) depends only on \(\sigma_\di\dx\) and the incident angle \(\theta\coloneqq\cos^{-1}(\nicefrac{\revise{\re(k_\di)}}{\omega})\). Consequently we always specify the value of the product \(\sigma_\di\dx\) for the damping coefficients (\ie~\(\sigma_\di\sim O(\nicefrac{1}{\dx})\)) independent of the grid size \(\dx\).  One also recognizes that Eq.~\eqref{eq:RhoAsymptotics} is the \emph{\((1,1)\)-Pad{\'e} approximation} of the exponential function
\begin{align}
	\rho(\sigma_\di,k_\di,\omega)\sim
	\exp\left(-\sigma_\di\dx\tfrac{k_\di}{\omega}\right),
    \revise{
    \quad
	|\rho(\sigma_\di,k_\di,\omega)|\sim
	\exp\left(-\sigma_\di\dx\tfrac{\re(k_\di)}{\omega}\right),
    }
\end{align}
which is the well-known decay rate of the continuous PML (\cf~\secref{sec:ComplexCoordinateStretching} or \cite[Eq.~(8)]{Johnson:2008:NPM}) with \(\int_{x_\di}^{x_\di+\dx}\sigma_\di(y)\, dy = \sigma_\di\dx\).
\revise{For the case where the PML is backed with a reflecting boundary, the relavant decay rate \eqref{eq:IntroReflectedDecayBound} also admits a similar asymptote
\begin{align}
    \frac{\rho(\sigma_\di,k_\di,\omega)}{\rho\left(\sigma_\di,-\conj{k_\di},\omega\right)}\sim
    \frac
    {\left(2-\sigma_\di\dx\frac{\re(k_\di)}{\omega}\right)^2+\left(\sigma_\di\dx\frac{\im(k_\di)}{\omega}\right)^2}
    {\left(2+\sigma_\di\dx\frac{\re(k_\di)}{\omega}\right)^2+\left(\sigma_\di\dx\frac{\im(k_\di)}{\omega}\right)^2},\quad
    \frac{|\rho(\sigma_\di,k_\di,\omega)|}{\left|\rho\left(\sigma_\di,-\conj{k_\di},\omega\right)\right|}
    \sim
    \exp\left(-2\sigma_\di\dx\tfrac{\re(k_\di)}{\omega}\right).
\end{align}
}

An issue that has been critically discussed in the literature is the dependency of the PML performance on the incident angle.  When \(\nicefrac{\revise{\re(k_\di)}}{\omega}\to 0\), (\ie~\(\theta\to 90^\circ\),) one finds \(\rho\to 1\) (for both discrete and continuous versions) and thus there is no damping effect.  This is nevertheless fine because such a wave will also take longer time before it penetrates the layer, not to mention the damping effect by the PMLs of other dimensional components \cite[Sec.~7.2]{Johnson:2008:NPM}. To be precise, for a PML with thickness \(\delta\), the residual wave takes \(\nicefrac{\delta}{\cos(\theta)} = \nicefrac{\delta\omega}{\revise{\re(k_\di)}}\) time to re-enter the physical domain.  In practice, by using a moderately thick PML, the residual waves are sufficiently damped and delayed so that they will not be measured (in the double precision) within a long given simulation time (\figref{fig:Roundtrip}~(c)).  Note that such an argument does not apply to other discrete PMLs in previous work due to their numerical reflections at the interface.

\begin{figure}[t]
	\centering
	\includegraphics{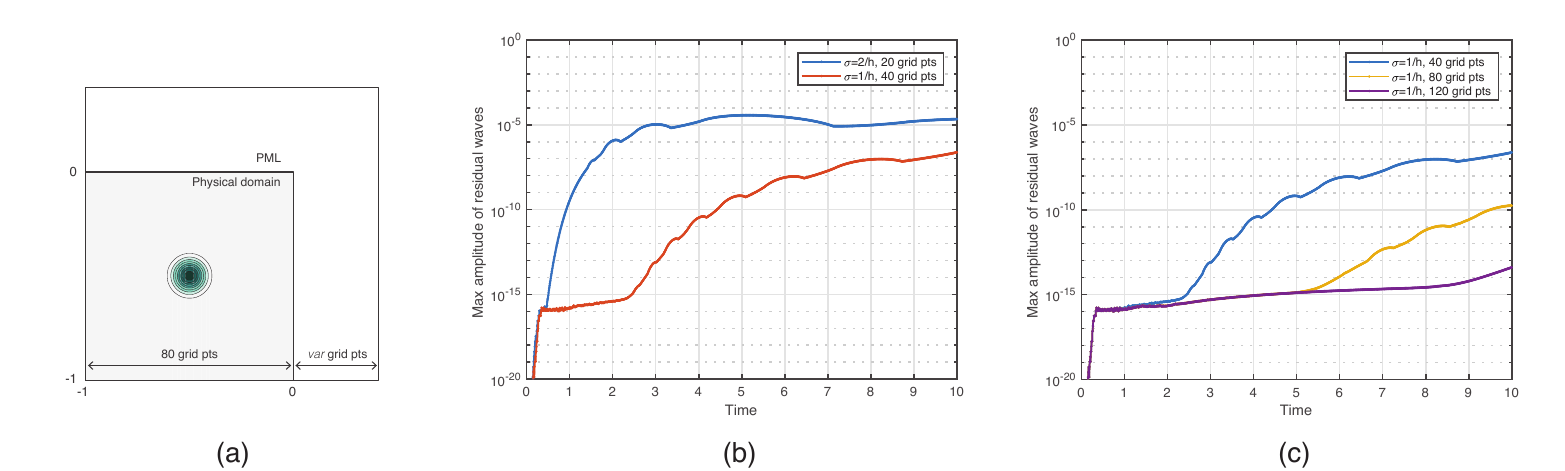}
	\caption{\label{fig:Roundtrip}Measurement of the residual waves. Eq.~\eqref{eq:DiscretePML} is simulated on the domain illustrated in (a) using a bump with peak value \(2\) as the initial condition. The constant damping coefficients and the thickness of the PML are specified in the legends  of (b), (c).  A remark on performance: Using a PML slightly greater than the domain size (\(\sigma=\nicefrac{1}{\dx}\), \(120\) grid points,) the solution of the PML system agrees in machine precision with the reference solution that has to be simulated on a domain \(10\) times greater in every dimension (wave speed \(=1\)).}
\end{figure}

How to choose the value \(\sigma_\di\dx\) for a good performance?  
In the smooth asymptotics, Eq.~\eqref{eq:RhoAsymptotics} suggests that \(\sigma_\di = \nicefrac{2}{\dx}\) gives the optimal performance for \(\theta = 0^\circ\) --- smooth waves with \(0^\circ\) incident angle are eliminated within one grid.  In the exact discrete formula, the norm of Eq.~\eqref{eq:IntroDecayRate} also assumes its minimum at \(s (= \sigma_\di) = \nicefrac{2}{\dx}\) among \([0,\nicefrac{2}{\dx}]\) for every fixed \((k_1,\ldots,k_d,\omega)\) satisfying the dispersion relation.  In principle, \(\sigma_\di = \nicefrac{2}{\dx}\) provides the optimal one-grid damping.  However, Eq.~\eqref{eq:IntroDecayRate} picks up a significant imaginary part that ``warps'' the complex phase of the solution \eqref{eq:IntroSolution}. For instance, when \(\sigma_\di\) is a constant, the component \(\rho^{\idx{i}_\di}\exp(\iu k_\di \dx \idx{i}_\di - \iu\omega t)\) in Eq.~\eqref{eq:IntroSolution} can be rewritten as \(|\rho|^{\idx{i}_\di}\exp(\iu(k_\di\dx + \arg\rho)\idx{i}_\di-\iu\omega)\). The modified group velocity emerging from the dispersion between \(\omega\) and \((k_\di\dx + \arg\rho)\) can cause the residual wave to re-enter the physical domain earlier.  Seeking a balance between rapid damping (\(\sigma_\di = \nicefrac{2}{\dx}\)) and large latency of the wave re-entry (\(\sigma_\di = 0\) which implies \(\arg(\rho)=0\)), one can simply take \(\sigma_\di = \nicefrac{1}{\dx}\).

\figref{fig:Roundtrip} demonstrates an open-space wave simulation in \([-1,0]^2\) with a simulation time interval \([0,10]\).  Let \(M\) denote the thickness of the PML in number of grid points.  \figref{fig:Roundtrip}~(b) shows that with the same amount of ``total damping'' \(\sum_{\idx{i}_\di=0}^{M-1}\sigma_\di(\idx{i}_\di)\dx\), the setup with \(\sigma_\di = \nicefrac{1}{\dx}\) gives a greater latency of the re-entering residual wave than the one with \(\sigma_\di = \nicefrac{2}{\dx}\).  Such a phenomenon should not be confused with the adiabatic effect due to a smoother profile of \(\sigma_\di\).  As shown in \figref{fig:Roundtrip}~(c), without any mesh refinement or rescaling of the value \(\sigma_\di = \nicefrac{1}{\dx}\), the residual wave is suppressed \emph{exponentially} as the thickness of PML increases \emph{linearly}.

In most applications the tolerance of the residual waves is at \(O(\dx^2)\), which is the size of the discretization error of the wave equation itself.  In those cases, as shown in \figref{fig:Roundtrip}, a thin PML with \(\sigma_\di = \nicefrac{2}{\dx}\) is sufficient.

\subsubsection{Stability}
\label{sec:Stability}
\begin{figure}[t]
	\centering
	\includegraphics[width=0.85\textwidth]{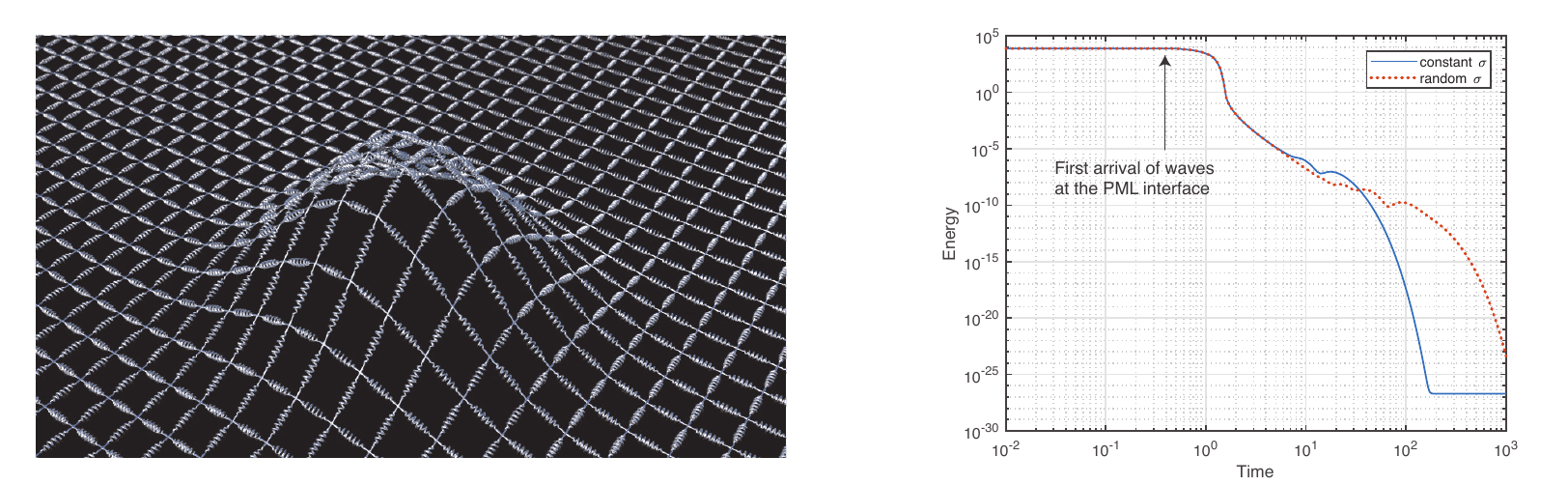}
	\caption{\label{fig:Energy}The discrete wave equation \eqref{eq:DiscreteWaveEquation} can be viewed as a spring mass system (left), providing a natural notion of total energy \eqref{eq:Energy}.  The plot (right) shows numerical stability by the long time behavior of the energy remained in the physical domain of \figref{fig:Roundtrip}~(a), with a constant \(\sigma=\nicefrac{2}{\dx}\), \(20\)-grid-point-thick PML, and a random \((\sigma_1(0),\ldots,\sigma_1(19),\sigma_2(0),\ldots,\sigma_2(19))\smash{\overset{\rm i.i.d.}{\sim}}\operatorname{Unif}([0,\nicefrac{2}{\dx}])\), \(20\)-grid-point-thick PML. }
\end{figure}
It is known that simulations of PML can exhibit temporal instability depending on the choice of damping coefficients \cite{Appelo:2006:PML}, how the PML is discretized \cite{Appelo:2006:NAL}, and how it is truncated into a finite domain \cite{Duru:2015:BCS}.  It is therefore necessary to check the numerical stability of Eq.~\eqref{eq:DiscretePML} truncated with periodic boundary.  The stability result is numerically confirmed by \figref{fig:Energy} by measuring the long time behavior of the energy of the system. Note that the system is stable even when the damping coefficients are random numbers.  The energy considered here is the sum of the kinetic and elastic energy associated to the discrete wave equation \eqref{eq:DiscreteWaveEquation} interpreted as a spring mass system (\figref{fig:Energy}~left):
\begin{align}
	\label{eq:Energy}
	\mathcal{E}\left[U,\tfrac{\partial U}{\partial t}\right](t) = \frac{1}{2}\sum_{\bi}\left|\frac{\partial U}{\partial t}(t,\bi)\right|^2 
	+ \frac{1}{2}\sum_{\bi}\sum_{\di=1}^d\left|\frac{(\tau_\di U)(t,\bi) - U(t,\bi)}{\dx}\right|^2,
\end{align} 
where the indices \(\bi\) iterate the sum over the grid points in the physical domain.  Note that this energy is conserved under \eqref{eq:DiscreteWaveEquation} when the physical domain is infinite.  On a truncated domain, the energy eventually decay to zero as all waves scatter away.  This phenomenon is stably captured by the truncated PML system \eqref{eq:DiscretePML} as shown in \figref{fig:Energy}~(right).

\subsubsection{Summary of Main Results}
A stable discrete PML \eqref{eq:DiscretePML} is proposed for the standardly discretized wave equation \eqref{eq:DiscreteWaveEquation}.  The PML creates no reflection at the interface at the discrete level.  This statement is given by both the explicit plane wave solutions \eqref{eq:IntroSolution} and numerical validations.   The numerically measured residual waves are suppressed exponentially below machine precision as the thickness of the layer grow linearly.  

To show the significance of the above results, we take a glimpse of the history of the related subject.

\subsection{Related Work}
\label{sec:RelatedWork}

Ever since numerical computations were possible, non-reflecting boundary treatments for wave equations have shown its importance in meteorology, seismology, acoustics, optics, electromagnetism, and quantum mechanics. Prior to the late 1970's, the common approach for wave absorbing boundary is adopting a ``one-sided difference'' (\ie~extrapolation), or imposing the \emph{Sommerfeld radiation boundary condition} \cite{Orlanski:1976:SBC}
\begin{align}
	\label{eq:Sommerfeld}
	\frac{\partial u}{\partial t} + \frac{\partial u}{\partial r} = 0\quad\text{at the boundary of the truncated domain,}
\end{align}
based on Sommerfeld's (1912) \cite{Sommerfeld:1912:GFS,Schot:1992:EYS} radiation asymptotics \(\nicefrac{\partial u}{\partial t} + \nicefrac{\partial u}{\partial r}\sim o(\nicefrac{1}{r^{(d-1)/2}})\) as \(r = |\bx|\to\infty\).  Unfortunately, Eq.~\eqref{eq:Sommerfeld} is reflectionless only for the one-dimensional problem.

In a multi-dimension setup, \citet{Fix:1978:VMU} generalized the Sommerfeld radiation boundary condition to a non-local condition, initiating a research area on reflectionless boundaries involving integral operators.  Non-local treatments can achieve the reflectionless property exactly in the continuous setup.  Methods of this class include explicitly removing the reflection by exploiting Kirchhoff's reflection formula \cite{Ting:1986:EBC} (non-local in time), or by projecting the solution of each time to the harmonics of the domain through boundary integrals or a Dirichlet-to-Neumann map \cite{Givoli:1990:NBC,Grote:1995:ENB,Sofronov:1998:ABC,Sofronov:2016:TTB} (non-local in space but local in time).   These non-local methods are also known as \emph{transparent boundary conditions}, which are seen today in, for example, computations of Schr\"odinger equations 
\cite{Antoine:2007:TAB,Mennemann:2014:PML}. 
The  \revise{main challenge} of the non-local methods is that  \revise{the exact evaluation of the integrals} are numerically expensive. 
\revise{Practical implementations require compressive numerical integrals such as multipole expansion and approximations of integral kernels \cite{Alpert:2000:REN, Alpert:2002:NBC,Jiang:2008:ERN} to reduce numerical complexity.}
From a strict perspective, the work known to the author on non-local exact boundary conditions has only reported ``substantial'' measured reflections compared with machine precision.

Another class of methods uses only local operators to improve the Sommerfeld boundary condition \eqref{eq:Sommerfeld}.  In 1977, Clayton, Engquist and Majda \cite{Clayton:1977:ABC,Engquist:1977:ABC,Engquist:1979:RBC} introduced the Absorbing Boundary Conditions (ABCs). To derive an ABC, one first writes a non-reflecting pseudodifferential boundary condition (which may be non-local in both space and time) based on the hyperbolic theory.  By truncating the Taylor or Pad\'e expansion of this pseudodifferential operator, one obtains a family of ABCs involving high-order differential operators.  Similar to the Engquist--Majda ABC family, \citet{Bayliss:1980:RBC} proposed a recursion operator that generates from \eqref{eq:Sommerfeld} an ever-improving hierarchy of boundary conditions.  ABCs were also derived for the discrete wave equation by \citet{Higdon:1986:ABC}.  ABCs are attractive due to their generality and numerical locality.  However,  any finite-order ABC is only an approximation towards reflectionless boundary condition even before discretization, unless the domain is 1D or the solution contains only finitely many harmonic components \cite{Hagstrom:1998:FAE}.   More non-reflecting boundary treatments can be found in the review articles by \citet{Givoli:1991:NBC,Givoli:2004:HLN} and \citet{Tsynkov:1998:NSP}.

In 1994, in the context of Maxwell's equations \citet{Berenger:1994:PML} introduced an artificial damping medium which has both electric and \emph{magnetic} conductivities, creating an effective impedance matching the vacuum.  This medium is a \emph{Perfectly Matched Layer} (PML). Attached to the boundary of a vacuum domain, PML becomes an ideal non-reflecting boundary treatment since it absorbs electromagnetic waves without creating any reflection at the interface. What is more, PML computation is local.  In the same year, PML was understood by \citet{Chew:1994:A3D} as a result of \emph{complex coordinate stretching}.  The technique of complex coordinate stretching is so general that it has taken PML to fluids and acoustics \cite{Hu:1996:ABC,Turkel:1998:APM}, elasticity and seismology \cite{Komatitsch:2003:PML}, waves on general coordinates \cite{Collino:1998:PML}, general linear hyperbolic systems \cite{Appelo:2006:PML}, Schr\"odinger equations \cite{Nissen:2011:OPM}, just to name a few.  Despite the success of PML, there has been rooms for improvement.  Until now, the main limitations of PML were the lost of perfect matching property after discretization and occasional numerical instability.

\citet{Berenger:1994:PML} reported that the discretized PML exhibits numerical reflections depending on the wave incident angles and damping coefficients, and thereby suggested a few profile functions for the damping coefficients.  One close attempt towards reflectionless discrete PML was given by \citet{Chew:1996:PML} by performing complex coordinate stretching to the already-discretized Maxwell's equations.  Unfortunately, without using Discrete Complex Analysis, numerical reflections remain, and optimization over damping parameters must be performed. 
With the objective of minimizing numerical reflection and attenuation rates, various discrete PML has been optimized over the damping coefficients  \cite{Fang:1996:CEN,Collino:1998:OPM,Winton:2000:SPC,Travassos:2006:OCP,Bermudez:2007:OPM,Nissen:2011:OPM}, stretching paths \cite{Becache:2004:LTB}, finite difference grids \cite{Asvadurov:2003:OFD}, and finite element meshes \cite{Chen:2003:AFE}. However, none of the known optimized PML comes close to eliminating numerical reflections entirely.

Numerical instabilities observed in a number of PMLs have brought about a continuously active research area. Early observations of instabilities were explained by the excessive number of auxiliary variables, which can take away strong well-posedness \cite{Abarbanel:1997:MAP}.  Stabilizing methods include introducing dissipation \cite{Hu:1996:ABC} and reformulating PMLs into ``unsplit'' forms (\emph{unsplit}-PML) \cite{Gedney:1996:APM,Petropoulos:2000:RSL,Abarbanel:2002:LTB} or convolution forms (\emph{convolutional}-PML) \cite{Roden:2000:CPM,Becache:2004:LTB,Komatitsch:2007:UCP}.  Later, the stability problem of PMLs are regarded as more complicated.  Mild instabilities were found even under the strong well-posedness condition \cite{Abarbanel:2002:LTB}, while classical split-form PMLs become stable again by using a \emph{multi-axial} PML \cite{Meza:2010:SNP}.  It is also discovered that severe instabilities can occur through backscattering in some anisotropic PMLs \cite{Becache:2003:SPM,Loh:2009:FRP}, leading to suggestions of adopting non-perfectly matching absorbers \cite{Oskooi:2008:FPM}.  Yet, by reformulation and carefully choosing the damping coefficients, researchers of \cite{Appelo:2006:NAL,Appelo:2006:PML,Becache:2017:SPM} showed that it is possible to regain stability for anisotropic PMLs.   More recent studies showed instabilities due to the far-end boundary of the PML, which are then stabilized by backing the PML with a dissipating boundary treatment \cite{Festa:2005:ISW,Deinega:2011:LTB,Duru:2015:BCS}.

\revise{
Along a similar discussion, the long-time performance of PML is more thoroughly analyzed through the complete spectral decomposition \cite{deHoop:2002:ABC,Diaz:2006:TDA,Hagstrom:2009:CRB}, \ie\ including the evanescent waves into the plane wave expansion \cite{Heyman:1996:TPS}.  Discretization based on complete spectral analysis can achieve long-time stability \cite{Chen:2012:LSC}.
}

A number of new approaches to non-reflecting boundaries are developed recently.  Hagstrom and coworkers \cite{Hagstrom:2014:DAB,Rabinovich:2015:DAB,Rabinovich:2017:DAB} propose a hybrid of ABC and PML named \emph{Double Absorbing Boundaries} (DAB).   Meanwhile, Druskin et~al.~\cite{Druskin:2013:KSC,Druskin:2014:EKS,Druskin:2016:NOP} improve PMLs by revisiting pseudodifferential operators and performing Krylov space analysis.  Another new absorbing boundary is designed by root-finding algorithms \cite{Lee:2018:ABC}.  However, a decent reflectionless boundary treatment with machine zero quality has not yet been achieved among these methods.

The present paper contributes a genuine discrete PML (Eq.~\eqref{eq:DiscretePML}) for the discrete wave equation that has no numerical reflection at all.  Residual waves can be suppressed exponentially just as the original continuous theory by \citet{Berenger:1994:PML}.  It is numerically stable, and remarkably simple to implement.

\section{Preliminaries}
As a preparation for the main derivation in \secref{sec:DerivationOfDiscretePML}, we recall the method of complex coordinate stretching for deriving the continuous PML equations, followed by some basic notions from Discrete Complex Analysis.
 
\subsection{Complex Coordinate Stretching}
\label{sec:ComplexCoordinateStretching}
We begin by applying Fourier transform in time to \eqref{eq:WaveEquation}:
\begin{align}
	\label{eq:WaveEquationFourier}
	-\omega^2\hat{u}(\omega,\bx) = \sum_{\di=1}^d\frac{\partial^2\hat u}{\partial x_\di^2}(\omega,\bx).
\end{align}
Then we \emph{complexify} the equation by extending the domain of the function \(\hat{u}(\omega,\cdot)\) from \(\RR^d\) to \(\CC^d\) and rewriting \eqref{eq:WaveEquationFourier} as
\begin{align}
	\label{eq:WaveEquationComplex}
	-\omega^2f(\bz) = \sum_{\di=1}^d\frac{\partial^2 f}{\partial z_\di^2}(\bz),\quad \frac{\partial f}{\partial \conj z_\di}(\bz) = 0,
\end{align}
where \(\bz = (z_1,\ldots,z_d)\) denotes the coordinate for \(\CC^d\), and \(\nicefrac{\partial}{\partial z_\di}\) (\resp~\(\nicefrac{\partial}{\partial\conj z_\di}\)) the holomorphic (\resp~anti-holomorphic) differential operator along each component.
Equation \eqref{eq:WaveEquationComplex} is an extension of \eqref{eq:WaveEquationFourier} in the sense that given any solution \(f\) to \eqref{eq:WaveEquationComplex} and real lines \(\ell_\di\colon\RR\hookrightarrow\CC\), \(\ell_\di\colon x\mapsto x+0\,\iu\), the composition \(\hat u(\omega,\bx)\coloneqq f(\ell_1(x_1),\ldots,\ell_d(x_d))\) is a solution to \eqref{eq:WaveEquationFourier}. Conversely, a solution \(\hat{u}(\omega,\bx)\) to \eqref{eq:WaveEquationFourier} gives rise to a solution \(f(\bz)\) to \eqref{eq:WaveEquationComplex} via analytic continuation. 
This follows from separating variables and the theory of linear differential equations with analytic coefficients (in this case constant).
In particular, the  \revise{traveling} wave solutions \(\exp(\iu\bk\cdot\bx)\) of \eqref{eq:WaveEquationFourier} are uniquely extended to
\begin{align}
	\label{eq:PlaneWave}
	&f_\bk(\bz) = \exp\left(\iu\bk\cdot\bz\right),\quad\bk = (k_1,\ldots,k_d)\in\revise{\CC^d},\quad\sum_{\di=1}^d k_{\di}^2 = \omega^2.
\end{align}

\begin{figure}
	\centering
	\includegraphics[width=0.5\textwidth]{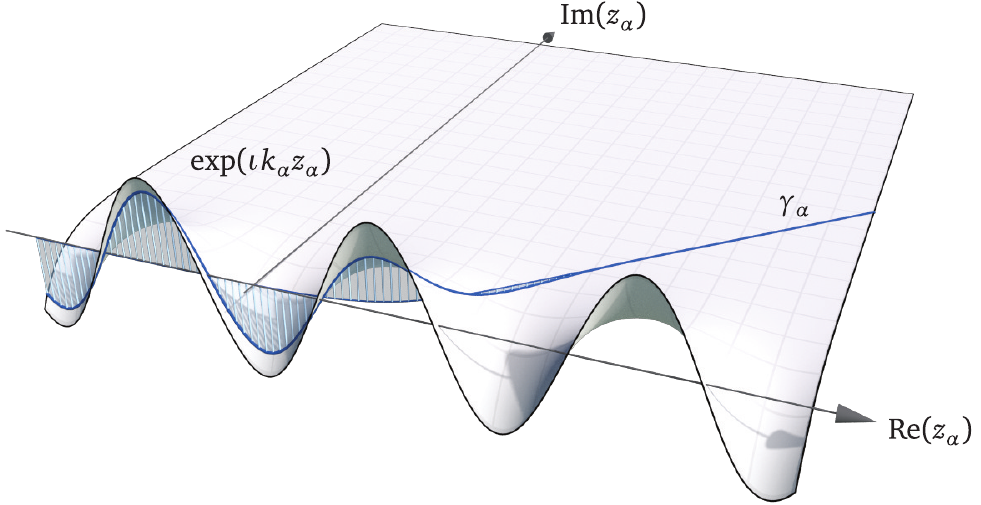}
	\caption{\label{fig:SmoothStretching}The analytic continuation of the plane wave, \(\CC\ni z_\di\mapsto\exp(\iu k_\di z_\di)\), is re-evaluated along a curve \(\gamma_\di\colon\RR\rightarrow\CC\) (Eq.~\eqref{eq:ContinuousCurve}).}
\end{figure}

Now, similar to the wave equation \eqref{eq:WaveEquationFourier} being the restriction of its complex extension \eqref{eq:WaveEquationComplex} on the real lines \(\ell_\di\), a PML equation is the restriction of \eqref{eq:WaveEquationComplex} on a set of more general curves
\begin{align}
	\label{eq:ContinuousCurve}
	\gamma_\di\colon\RR\rightarrow\CC,\quad\gamma_\di(x) = x+\tfrac{b_\di(x)}{\omega}\iu,\quad b_\di(x)\vert_{x\leq 0} \equiv 0,
\end{align}
where \(b_\di\colon\RR\rightarrow\RR\) are some non-decreasing absolutely continuous functions (\figref{fig:SmoothStretching}).  The purpose of the factor \(\nicefrac{1}{\omega}\) will be clear in a moment.  Substituting \eqref{eq:PlaneWave}, we have the plane wave solution \(\hat{u}_{\bk}(\omega,\bx) = f_\bk(\gamma_1(x_1),\ldots,\gamma_d(x_d))\) evaluated as
\begin{align}
	\label{eq:ContinuousDecayRate}
	\hat{u}_\bk(\omega,\bx) &= \exp\left(\iu\sum_{\di=1}^d k_\di\Big(x_\di+\tfrac{b_\di(x_\di)}{\omega}\iu\Big)\right)
	= \exp\left(-\sum_{\di=1}^d \tfrac{k_\di}{\omega} b_\di(x_\di)\right)\exp(\iu\bk\cdot\bx).
\end{align}
Since \(b_\di\) is non-decreasing, the scaling factor \(\exp(-\sum_{\di=1}^d(\nicefrac{\revise{\re(k_\di)}}{\omega})b_\di(x_\di))\) decays as (i) \(x_\di\) increases when \(\nicefrac{\revise{\re(k_\di)}}{\omega}>0\); and (ii) \(x_\di\) decreases when \(\nicefrac{\revise{\re(k_\di)}}{\omega}<0\). That is, any traveling wave is attenuated along its traveling direction. Moreover, the decay rate  \((\nicefrac{\revise{\re(k_\di)}}{\omega})b_\di'(x_\di)\) only depends on the parameter \(b_\di'\) and the wave direction \(\nicefrac{\revise{\re(\bk)}}{|\bk|}=\nicefrac{\revise{\re(\bk)}}{|\omega|}\); in particular it does not depend on the scale of the frequency \(|\bk|=|\omega|\).  The parameter \(\sigma_\di(x)\coloneqq b_\di'(x)\) (a nonnegative locally integrable function) is the \emph{PML damping coefficient}.  Note that the attenuated wave is just the original plane wave evaluated differently on the complex domain. In particular, when a  wave travels across regions with varying damping coefficients, no reflection is produced.

Finally, the PML equations are derived by changing variables for \eqref{eq:WaveEquationComplex}. Holomorphicity of \(f\) yields
\begin{align*}
	\frac{\partial f}{\partial z_\di}(\gamma_1(x_1),\ldots,\gamma_d(x_d)) &= \frac{1}{\gamma_\di'(x_\di)}\frac{\partial}{\partial x_\di}\big(f(\gamma_1(x_1),\ldots,\gamma_d(x_d))\big)
	=\frac{1}{1+\frac{\iu}{\omega}\sigma_\di(x_\di)}\frac{\partial}{\partial x_\di}\big(f(\gamma_1(x_1),\ldots,\gamma_d(x_d))\big).
\end{align*}
Therefore, by substituting \(\hat u(\omega,\bx) = f(\gamma_1(x_1),\ldots,\gamma_d(x_d))\), we have
\begin{align*}
	\frac{\partial f}{\partial z_\di}(\gamma_1(x_1),\ldots,\gamma_d(x_d)) &=
	\frac{1}{1+\frac{\iu}{\omega}\sigma_\di(x_\di)}\frac{\partial\hat u}{\partial x_\di}(\omega,\bx).
\end{align*}
Again by holomorphicity of \(\nicefrac{\partial f}{\partial z_\di}\), 
\begin{align*}
	\frac{\partial^2 f}{\partial z_\di^2}(\gamma_1(x_1),\ldots,\gamma_d(x_d)) = 
	\frac{1}{1+\frac{\iu}{\omega}\sigma_\di(x_\di)}\frac{\partial}{\partial x_\di}
	\left(
	\frac{1}{1+\frac{\iu}{\omega}\sigma_\di(x_\di)}\frac{\partial}{\partial x_\di}\hat{u}(\omega,\bx)
	\right).
\end{align*}
Therefore, Eq.~\eqref{eq:WaveEquationComplex} restricted along the curves \(\gamma_\di\) is given by
\begin{align}
	\label{eq:PMLDerivation0}
	-\omega^2\hat{u} = \sum_{\di=1}^d\frac{1}{1+\frac{\iu}{\omega}\sigma_\di(x_\di)}\frac{\partial}{\partial x_\di}
	\left(
	\frac{1}{1+\frac{\iu}{\omega}\sigma_\di(x_\di)}\frac{\partial\hat u}{\partial x_\di}
	\right).
\end{align} 

Now, the remaining steps are to rearrange \eqref{eq:PMLDerivation0} into a form which can be transformed back to the time domain.  Dividing both sides of Eq.~\eqref{eq:PMLDerivation0} by \(-\iu\omega\) yields
\begin{align}
	\label{eq:PMLDerivation1}
	-\iu\omega\hat u &=
	\sum_{\di=1}^d\frac{1}{1+\frac{\iu}{\omega}\sigma_\di(x_\di)}\frac{\partial}{\partial x_\di}
		\left(
		\frac{1}{-\iu\omega+\sigma_\di(x_\di)}\frac{\partial\hat u}{\partial x_\di}
		\right)
		=
		-\sum_{\di=1}^d \frac{1}{1+\frac{\iu}{\omega}\sigma_\di(x_\di)}\frac{\partial\hat v_\di}{\partial x_\di}
\end{align}
where we define 
\begin{align}
	\label{eq:VHatDefinition}
	\hat v_\di(\omega,\bx) \coloneqq  \frac{1}{\iu\omega}\frac{\partial f}{\partial z_\di}(\gamma_1(x_1),\ldots,\gamma_d(x_d))
	= -\frac{1}{-\iu\omega+\sigma_\di(x_\di)}\frac{\partial\hat u}{\partial x_\di}(\omega,\bx),
\end{align} 
or equivalently, we impose the equations
\begin{align}
	\label{eq:PMLMomentumEquationFourier}
	-\iu\omega\hat v_\di = -\sigma_\di(x_\di)\hat v_\di - \frac{\partial\hat u}{\partial x_\di},\quad \di = 1,\ldots,d.
\end{align}
Using the algebraic fact 
\[\frac{1}{1+\frac{\iu}{\omega}\sigma_\di(x_\di)} = 1-\frac{\sigma_\di(x_\di)}{-\iu\omega + \sigma_\di(x_\di)},\] 
rewrite Eq.~\eqref{eq:PMLDerivation1} as
\begin{align}
	\label{eq:PMLDerivation2}
	-\iu\omega\hat{u} &= -\sum_{\di=1}^d\frac{\partial\hat v_\di}{\partial x_\di} + \sum_{\di=1}^d\frac{\sigma_\di(x_\di)}{-\iu\omega + \sigma_\di(x_\di)}\frac{\partial\hat v_\di}{\partial x_\di}
	=-\sum_{\di=1}^d\frac{\partial\hat v_\di}{\partial x_\di} +\sum_{\di=1}^d \hat q_\di
\end{align}
by introducing auxiliary variables 
\[\hat q_\di(\omega,\bx)\coloneqq \frac{\sigma_\di(x_\di)}{-\iu\omega + \sigma_\di(x_\di)}\frac{\partial\hat v_\di}{\partial x_\di}(\omega,\bx),\] 
or equivalently,
\begin{align}
	\label{eq:PMLAuxEquationFourier}
	-\iu\omega\hat q_\di = -\sigma_\di(x_\di)\hat q_\di + \sigma_\di(x_\di)\frac{\partial\hat v_\di}{\partial x_\di}.
\end{align}
Combining \eqref{eq:PMLDerivation2}, \eqref{eq:PMLMomentumEquationFourier} and \eqref{eq:PMLAuxEquationFourier} and transforming them into the time domain, we arrive at the PML equations:
\begin{subequations}
	\label{eq:PMLEquation}
\begin{align}[left=\empheqlbrace]
	    \label{eq:PMLEquationA}
		&\frac{\partial u}{\partial t} = -\sum_{\di=1}^d \frac{\partial v_\di}{\partial x_\di} + \sum_{\di=1}^d q_\di\\
		\label{eq:PMLEquationB}
		&\frac{\partial v_\di}{\partial t} = -\sigma_\di(x_\di) v_\di - \frac{\partial u}{\partial x_\di}\\
		\label{eq:PMLEquationC}
		&\frac{\partial q_\di}{\partial t} = -\sigma_\di(x_\di) q_\di + \sigma_\di(x_\di)\frac{\partial v_\di}{\partial x_\di}.
\end{align}
\end{subequations}

In a region where \(\sigma_\di = 0\) for all \(\di=1,\ldots,d\), the auxiliary variables \(q_\di(t,\bx)\) is set to zero; in this case Eq.~\eqref{eq:PMLEquation} becomes the acoustic wave equation (linearized barotropic Euler equations)
\begin{align*}[left=\empheqlbrace]
		&\frac{\partial u}{\partial t} = -\sum_{\di=1}^d\frac{\partial v_\di}{\partial x_\di}\\
		&\frac{\partial v_\di}{\partial t} = -\frac{\partial u}{\partial x_\di},
\end{align*}
which is equivalent to the scalar wave equation \eqref{eq:WaveEquation}.  In particular, from the acoustics viewpoint, \(v_\di\) is equipped with a physical interpretation of \emph{material velocity}.  

Since the wave equation is just the spacial case of the PML equations with zero damping, and a spatially varying damping coefficient generates no reflection, the PML equations \eqref{eq:PMLEquation} ``perfectly match'' the wave equation and absorb all incident waves.

Readers who hasten to compare Eq.~\eqref{eq:PMLEquation} with Eq.~\eqref{eq:DiscretePML} may not find consistency obvious.  Note that both \(\Phi_\di(t,\bi)\) and \(\Psi_\di(t,\bi)\) approximate \(v_\di(t,\bx)\), as Eq.~\eqref{eq:DiscretePMLb} and Eq.~\eqref{eq:DiscretePMLc} are both discretizations of Eq.~\eqref{eq:PMLEquationB}. Also, Eq.~\eqref{eq:PMLEquationA} is in a first-order form while Eq.~\eqref{eq:DiscretePMLa} is second order in time. Nonetheless, a clear analogy between Eq.~\eqref{eq:DiscretePML} and Eq.~\eqref{eq:PMLEquation} exists in the complex domain, which we will see in \secref{sec:DerivationOfDiscretePML}.

\subsection{Discrete Complex Analysis}
\label{sec:DiscreteComplexAnalysis}

Here let me introduce a few basic notions in Discrete Complex Analysis. In this study, the domains of complex-valued functions are replaced by quadrilateral graphs in \(\CC\). With the ``right'' definitions for discrete holomorphicity and discrete differential forms, one discovers a profound theory where discrete versions of Cauchy's Integral Theorem, conformality results, \etc, hold in the exact sense.  For our application to discrete PMLs, we take only the essential definitions that allow us to calculate the analytic continuation of discrete plane waves.  Without exploiting the full theory of Discrete Complex Analysis, the readers are pointed to \cite{Duffin:1956:BPD,Bobenko:2005:LNT,Lovasz:2004:DAF,Bobenko:2016:DCA} for survey and further study.

For simplicity we consider an infinite quadrilateral lattice \(\Lambda\) indexed by \(\ZZ^2\). The lattice is described by a set of vertices \(V(\Lambda) = \ZZ^2\), a set of edges
 \(E(\Lambda) = \{\idx{e_{i+\half,j}},\idx{e_{i,j+\half}}\,\vert\,(\idx{i,j})\in\ZZ^2\}\) where
\begin{align*}
	\idx{e_{i+\half,j}} \coloneqq \big((\idx{i,j}),(\idx{i+1,j})\big),\quad
	\idx{e_{i,j+\half}} \coloneqq \big((\idx{i,j}),(\idx{i,j+1})\big)
\end{align*}
are the horizontal and vertical edges, and a set of faces \(F(\Lambda) = \{Q_\idx{i+\half,j+\half}\,\vert\,(\idx{i,j})\in\ZZ^2\}\) where each \(Q_\idx{i+\half,j+\half}\) is the quadrilateral
\begin{align*}
	Q_\idx{i+\half,j+\half} \coloneqq \big((\idx{i,j}),(\idx{i+1,j}),(\idx{i+1,j+1}),(\idx{i,j+1})\big).
\end{align*}
The above combinatorial lattice can be realized as a graph in \(\CC\) by assigning vertex positions \(z\colon V(\Lambda)\rightarrow\CC\).  Through linear interpolation the edges in \(E(\lambda)\) are mapped to straight lines connecting the vertex positions. 
\revise{Note that these images of edges can cross or overlap in general.}
The lattice \(\Lambda\) with the structure \(z= (z_\idx{i,j})_{(\idx{i,j})\in\ZZ^2}\) 
characterizes a \emph{discrete complex domain} \((\Lambda,z)\).

A complex-valued function \(f\) on \(\Lambda\) is simply another quadrilateral graph realized in \(\CC\) given by values assigned to vertices \( f = (f_\idx{i,j})_{(\idx{i,j})\in\ZZ^2}\).

\begin{figure}
	\centering
	\includegraphics[width=0.4\textwidth]{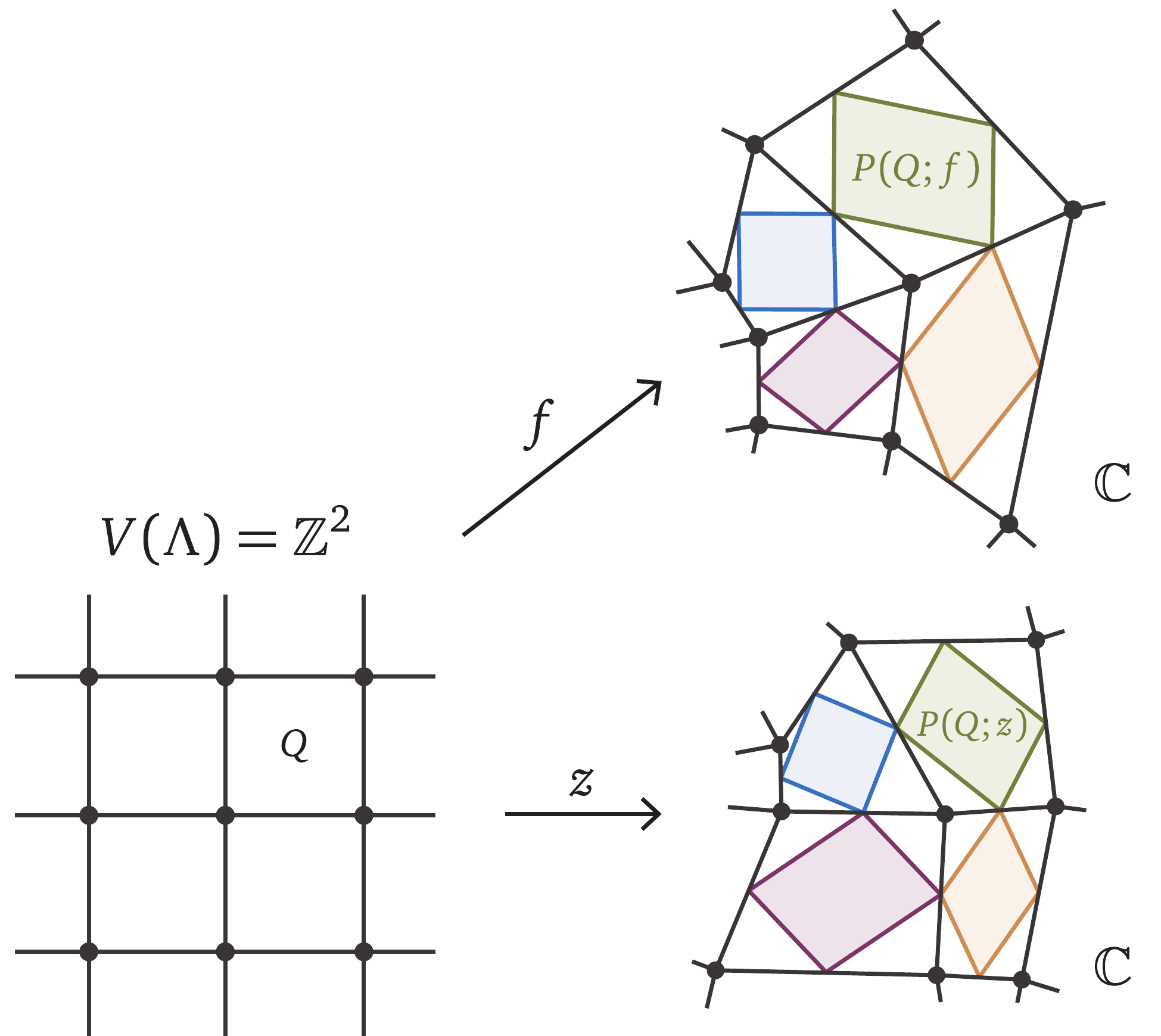}
	\caption{\label{fig:DCA}A function \(f\colon V(\Lambda)\rightarrow\CC\) is discretely holomorphic with respect to a complex domain \((\Lambda,z)\), \(z\colon V(\Lambda)\rightarrow\CC\), if each of the medial parallelograms in the graph mapped by \(f\) is a scaled rotation of the corresponding one mapped by \(z\).}
\end{figure}

\begin{definition}[Discrete holomorphicity]
	\label{def:ComplexDifferentiable}
	A complex-valued function \(f\) on \(\Lambda\) is said to be \emph{complex-differentiable with respect to \(z\)} at a face \(Q_\idx{i+\half,j+\half}\in F(\Lambda)\) if \(f\) satisfies the \emph{discrete Cauchy--Riemann equation}
	\begin{align}
		\label{eq:DiscreteCauchyRiemann}
		\frac{f_\idx{i+1,j+1}-f_\idx{i,j}}{z_\idx{i+1,j+1}-z_\idx{i,j}} = 
		\frac{f_\idx{i,j+1}-f_\idx{i+1,j}}{z_\idx{i,j+1}-z_\idx{i+1,j}}.
	\end{align}
	In this case, the \emph{discrete complex derivative} \(D_zf (Q_\idx{i+\half,j+\half})\) is defined as the value of \eqref{eq:DiscreteCauchyRiemann}.
	The function \(f\) is said to be \emph{discretely holomorphic} on \((\Lambda,z)\) if \(f\) is complex-differentiable with respect to \(z\) at every face in \(F(\Lambda)\).
\end{definition}

A clear geometric picture for \defref{def:ComplexDifferentiable} will be given by \propref{prop:GeometricHolomorphicity} in terms of the conformality of the \emph{medial parallelogram} of each face (\figref{fig:DCA}).

\begin{definition}[Medial parallelogram]
	For a given quadrilateral graph \((\Lambda,z)\), let \(z_\idx{i+\half,j} = {(z_\idx{i,j}+z_\idx{i+1,j})}\nicefrac{}{2}\) and \(z_\idx{i,j+\half} = (z_\idx{i,j}+z_\idx{i,j+1})\nicefrac{}{2}\) denote the midpoint on each edge in \(E(\Lambda)\). For each face \(Q_\idx{i+\half,j+\half}\in F(\Lambda)\), the \emph{medial parallelogram} is given by connecting the midpoints of the incident edges
	\begin{align*}
		P(Q_\idx{i+\half,j+\half};z)\coloneqq \big(z_\idx{i+\half,j},z_\idx{i+1,j+\half},z_\idx{i+\half,j+1},z_\idx{i,j+\half}\big).
	\end{align*}
\end{definition}

Note that \(P(Q_\idx{i+\half,j+\half};z)\) is always a parallelogram, since its four edge vectors are pairwise given explicitly as \((z_\idx{i+1,j+1} - z_\idx{i,j})\nicefrac{}{2}\) and \((z_\idx{i,j+1} - z_\idx{i+1,j})\nicefrac{}{2}\).  

\begin{proposition}
	\label{prop:GeometricHolomorphicity}
	A complex-valued function \(f\) on \(\Lambda\) is complex-differentiable with respect to \(z\) at \(Q\in F(\Lambda)\) if and only if the medial parallelogram \(P(Q;f)\) is a scaled rotation of \(P(Q;z)\). When both statements hold, the latter scaled rotation, described as a complex number, equals to the discrete complex derivative \(D_zf(Q)\). 
\end{proposition}

In the next section, we use the notion of discrete holomorphicity to study a discrete analog of analytic continuation of plane waves, and from there we derive the discrete PML equations \eqref{eq:DiscretePML}.

\section{Derivation of the Discrete PML Equations}
\label{sec:DerivationOfDiscretePML}

In this section, we derive the PML equations \eqref{eq:DiscretePML} for the discrete wave equation \eqref{eq:DiscreteWaveEquation}. 
Similar to \secref{sec:ComplexCoordinateStretching}, we set out the derivation by applying the Fourier transform in \(t\) to \eqref{eq:DiscreteWaveEquation}:
\begin{align}
	\label{eq:DiscreteWaveEquationFourier}
	-\omega^2\hat U = \sum_{\di=1}^d \frac{1}{\dx^2}\left(-2\hat{U} + \tau_\di^{-1}\hat U + \tau_\di\hat U\right),\quad\hat U = \hat U(\omega,\bi),\quad\bi=(\idx{i}_1,\ldots,\idx{i}_d)\in\ZZ^d.
\end{align}
The  \revise{elementary} wave solutions to \eqref{eq:DiscreteWaveEquationFourier}
\revise{bounded in the half space \(\bigcap_{\di=1}^d\{\idx{i}_\di<0\}\)}
 are given by the modal functions \(\hat{U}_{\bk}(\bi)\) defined for wave numbers \(\bk=(k_1,\ldots,k_d)\in \revise{\cK^d}\) \revise{(\cf~\figref{fig:KDomain})} as
\begin{align}
	\label{eq:DiscreteMode}
	\hat{U}_{\bk}(\bi) = \prod_{\di=1}^d W_{k_\di}(\idx{i}_\di),\quad 
	W_{k}(\idx{i})\coloneqq 
		e^{\iu k \dx\idx{i}}.
\end{align}
%
For each \(\di=1,\ldots,d\), 
\(W_{k_\di}(\idx{i}_\di)\)
is an eigenvector of the difference operator \((-2+\tau_\di^{-1}+\tau_\di)/\dx^2\) with eigenvalue \(-\nicefrac{4}{\dx^2}\sin^2(\nicefrac{k_\di\dx}{2})\), turning the difference equation \eqref{eq:DiscreteWaveEquationFourier} into a dispersion relation 
\begin{align}
	\label{eq:DiscreteDispersion}
	\omega^2\dx^2 = \sum_{\di=1}^d 4\sin^2\left(\tfrac{ k_\di\dx}{2}\right).
\end{align}

The goal is to find a ``complexification'' of \eqref{eq:DiscreteWaveEquationFourier}.
That is, we seek a difference equation on a discrete complex domain whose normal modes are the \emph{discrete analytic continuation} of \eqref{eq:DiscreteMode}.
More precisely, for each dimensional component \(\di\) of the domain \(\ZZ^d\), we will construct a discrete complex domain \((\Lambda,z^{(\di)})\) that supports discrete functions of \(d\) complex variables
\(
	f\colon (V(\Lambda))^d\rightarrow\CC,
\)
and extend the difference operator \((-2+\tau_\di^{-1}+\tau_\di)/\dx^2\) so that it operates on \(f\).  Once this step is established, we pose the \emph{complexified discrete wave equation}:
\begin{subequations}
\label{eq:ComplexifiedDiscreteWave}
	\begin{align}
		\label{eq:ComplexifiedDiscreteWaveA}
		&-\omega^2 f = \sum_{\di=1}^d \frac{1}{\dx^2}\left(-2 + \tau_\di^{-1} + \tau_\di\right) f,\quad f\colon (V(\Lambda))^d\rightarrow\CC,\\
		\label{eq:ComplexifiedDiscreteWaveB}
		&(\idx{i}_\di,\idx{j}_\di)\mapsto f\left((\idx{i}_1,\idx{j}_1),\ldots,(\idx{i}_\di,\idx{i}_\di),\ldots(\idx{i}_d,\idx{j}_d)\right)\text{ is discretely holomorphic with respect to \(z^{(\di)}\)}.
	\end{align}
\end{subequations}
By separation of variables, the normal modes of Eq.~\eqref{eq:ComplexifiedDiscreteWave} are given by the holomorphic eigenfunctions of the extended difference operator \((-2+\tau_\di^{-1}+\tau_\di)/\dx^2\).  As will be shown, each plane wave \eqref{eq:DiscreteMode} is uniquely continued into one of these holomorphic eigenfunctions.  From there, we can restrict the equation to a path in analogy to \secref{sec:ComplexCoordinateStretching} and arrive at a set of PML equations.

 Before we reprise the complexified discrete wave equation \eqref{eq:ComplexifiedDiscreteWave} and the rest of the derivations in \secref{sec:DiscreteComplexCoordinateStretching}, we devote \secref{sec:DiscreteAnalyticContinuationOfPlaneWaves} to developing the  \revise{traveling} wave theory on a specifically designed discrete complex domain.

\subsection{Discrete Analytic Continuation of \revise{Traveling} Waves}
\label{sec:DiscreteAnalyticContinuationOfPlaneWaves}
For clarity of exposition here and in \secref{sec:StretchingPath}, we focus only on the \(\di^\text{th}\) direction of the original domain \(\ZZ^d\).  
  We will drop the subscript \(\di\) whenever it causes no confusion.  At times we use the notation \(\CC^D = \operatorname{Map}(D;\CC) = \{\varphi\colon D\rightarrow\CC\}\) denoting the space of all \(\CC\)-valued functions defined on a domain \(D\). With the obvious notion, \(\CC^D\) is an algebra over \(\CC\) with pointwise multiplication.

We consider the lattice \(\Lambda\), \(V(\Lambda) = \ZZ^2\), as described in \secref{sec:DiscreteComplexAnalysis}, and design a set of vertex positions \(z\colon V(\Lambda)\rightarrow\CC\),
\begin{align}
	\label{eq:VertexPositions}
	z_\idx{i,j} = (\idx{i+j})\dx + \iu b_\idx{j},
\end{align}
where \(b_\idx{j}\) is a real-valued monotone sequence given in terms of nonnegative real numbers \(s_\idx{j}\) as
\begin{align*}
	b_\idx{j} = \sum_{\tilde{\idx{j}}=0}^{\idx{j}-1}\frac{s_{\tilde{\idx{j}}}\dx}{\omega}.
\end{align*}
The parameters \(s_\idx{j}\) controlling the ``vertical gaps'' in the quadrilateral graph will later become the PML damping coefficients.  See \figref{fig:ParallelogramGraph}.

\begin{figure}
	\centering
	\includegraphics[width=0.8\textwidth]{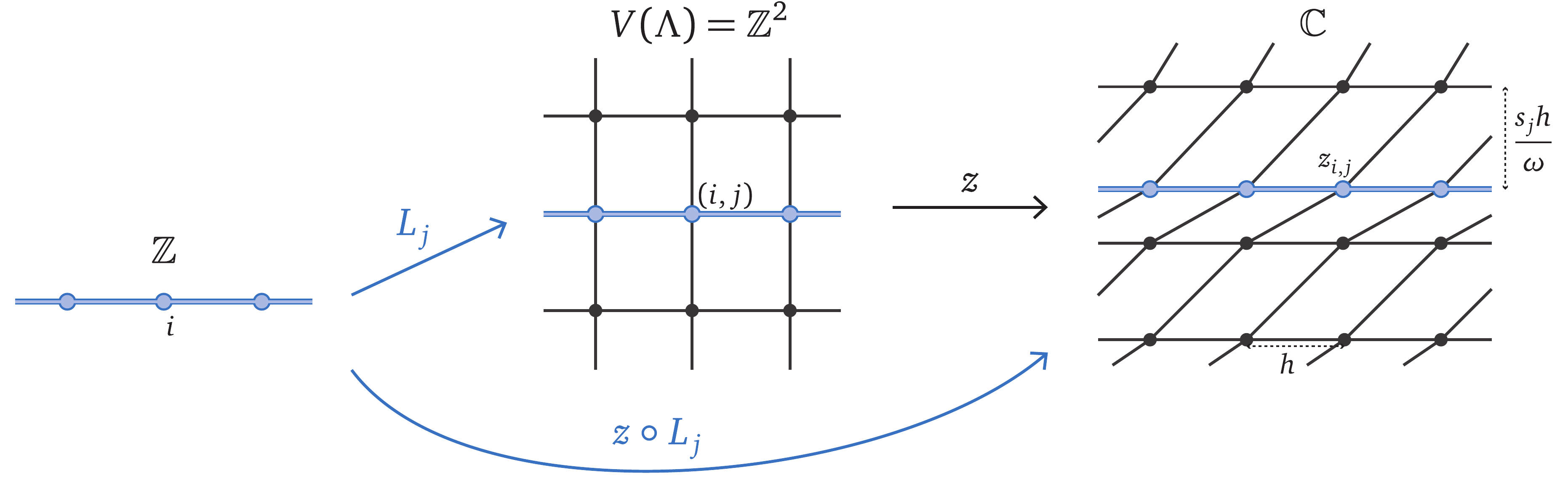}
	\caption{\label{fig:ParallelogramGraph}The discrete complex domain \((\Lambda,z)\) described by \eqref{eq:VertexPositions} is a parallelogram graph in \(\CC\). The embedded horizontal lines \(L_\idx{j}\colon\idx{i}\mapsto (\idx{i},\idx{j})\) are mapped to lines parallel to the real axis by \(z\).}
\end{figure}

Now, observe that the quadrilateral graph \((\Lambda,z)\) has each of its horizontal edges being the real constant \(\dx = z_\idx{i+1,j}-z_\idx{i,j}\).  In particular, \((\Lambda,z)\) is foliated into embedded real lines \(z\circ L_\idx{j}\) denoted by the embeddings \(L_\idx{j}\colon\ZZ\hookrightarrow V(\Lambda)\), \(\idx{i}\mapsto(\idx{i},\idx{j})\).  
With a distinguished horizontal direction, the shift operator \(\tau\) is naturally defined for complex-valued functions \(f\colon V(\Lambda)\rightarrow\CC\):
\begin{align*}
	(\tau f)_\idx{i,j} = f_\idx{i+1,j},
\end{align*}
which is compatible with the shift operators defined on \(\CC^\ZZ\) when restricted to the horizontal lines:
\[
	(\tau^{\pm 1}f)\circ L_\idx{j} = \tau^{\pm 1}(f\circ L_\idx{j}).
\]
By the same token, the difference operator \((-2+\tau^{-1}+\tau)/\dx^2\) is extended to \(\CC^{V(\Lambda)}\) and commutes with the pullback by \(L_\idx{j}\).
The latter \emph{naturality} property allows us to identify the eigenfunctions of \((-2+\tau^{-1}+\tau)/\dx^2\) on \(\CC^{V(\Lambda)}\) easily.  Suppose \(\mathcal{Y}_\lambda\subset\CC^\ZZ\) denotes the 
 eigenspace of \((-2+\tau^{-1}+\tau)/\dx^2\vert_{\CC^\ZZ}\) corresponding to an eigenvalue \(\lambda\).  Then an eigenfunction \(f\) of \((-2+\tau^{-1}+\tau)/\dx^2\vert_{\CC^{V(\Lambda)}}\) for the eigenvalue \(\lambda\) is in the most general form given by \(f_\idx{i,j} = Y_\idx{j}(\idx{i})\) with arbitrary choices of \((Y_\idx{j}\in\mathcal{Y}_\lambda)_{\idx{j}\in\ZZ}\) for each horizontal line.
 %
Fortunately, this bewilderingly large eigenspace is reduced to a much smaller space by imposing the discrete holomorphicity (\cf~\defref{def:ComplexDifferentiable}).

\begin{theorem}
	\label{thm:DecayRate}
	For each \(k\in \revise{\cK}\), \(\omega\in\RR\), \(\omega\neq 0\),
	there is a unique discrete holomorphic (with respect to \(z\)) eigenfunction \(f_k\colon V(\Lambda)\rightarrow\CC\) of \((-2+\tau^{-1}+\tau)/\dx^2\) such that \(f_k\circ L_0 = W_k\) (as defined in \textup{Eq.}~\eqref{eq:DiscreteMode}). Explicitly, for \(\idx{j}> 0\)
	\begin{align}
		\label{eq:DecayRate}
		(f_k)_\idx{i,j} = \left(\prod_{\tilde{\idx{j}}=0}^{\idx{j}-1}\rho({s_{\tilde{\idx{j}}}},k,\omega)\right)W_k(\idx{i}+\idx{j}),\quad \rho(s,k,\omega) \coloneqq \frac{2+\iu\frac{s}{\omega}(1-e^{-\iu k\dx})}{2+\iu\frac{s}{\omega}(1-e^{\iu k\dx})},
	\end{align}
	and for \(\idx{j}<0\)
	\begin{align}
		(f_k)_\idx{i,j} = \left(\prod_{\tilde{\idx{j}}=\idx{j}}^{-1}\rho({s_{\tilde{\idx{j}}}},k,\omega)^{-1}\right)W_k(\idx{i}+\idx{j}).
	\end{align}
\end{theorem}
\begin{proof}
	See \ref{app:ProofOfThmDecayRate}.
\end{proof}

The factor \(\rho(s_\idx{j},k,\omega)\) is a discrete analog of the the exponential factor of the analytic continuation of a plane wave (\cf~Eq.~\eqref{eq:ContinuousDecayRate}).  
\revise{The following theorem guarantees that}
 \(|\rho(s,k,\omega)|<1\) whenever \(k,\omega\) represent a right-traveling wave and \(s>0\).

\begin{theorem}
	\label{thm:DecayRateLessThanOne}
	For every \(s>0\), \(\omega\in\RR\), \(\omega\neq 0\), and \revise{\(k\in\cK\) with \(0<\sgn(\omega)\re(k)<\nicefrac{\pi}{\dx}\)}, we have \(|\rho(s,k,\omega)|<1\) \revise{and \(\nicefrac{|\rho(s,k,\omega)|}{|\rho(s,-\conj{k},\omega)|}<1\)}.
\end{theorem}
\begin{proof}
    \revise{See \ref{app:ProofOfThmDecayRateLessThanOne}.}
\end{proof}

Another important case of \thmref{thm:DecayRate} is that \(\rho(0,k,\omega) = 1\).  In particular, \((f_k)_\idx{i,j} = W_k(\idx{i}+\idx{j})\) whenever \(f\circ L_0 = W_k\) and \(s_{\tilde{\idx{j}}} = 0\) for all \(\tilde{\idx{j}}\) between \(\idx{j}\) and \(0\).  In fact, a vanishing vertical gap \(s_{\idx{j}}=0\) leads to the following general statement.
\begin{lemma}
	\label{lem:ZeroGap}
	Suppose \(f\) is a discrete holomorphic function on a discrete complex domain \((\Lambda,z)\) described by \eqref{eq:VertexPositions} with \(s_{\idx{j}} = 0\) for some \(\idx{j}\in\ZZ\). Then
	\begin{align}
		\tau(f\circ L_\idx{j}) = (\tau f)\circ L_\idx{j} = f\circ L_\idx{j+1}.
	\end{align}
\end{lemma}
\begin{proof}
	Eq.~\eqref{eq:DiscreteCauchyRiemann} with \(z_{\idx{i,j+1}} - z_{\idx{i+1,j}} = {\iu s_\idx{j}\dx}/{\omega} = 0\) yields \(f_\idx{i,j+1} - f_\idx{i+1,j}=0\) for all \(\idx{i}\in\ZZ\).
\end{proof}
\subsection{The Stretching Path}
\label{sec:StretchingPath}

\begin{figure}
	\centering
	\includegraphics[width=0.82\textwidth]{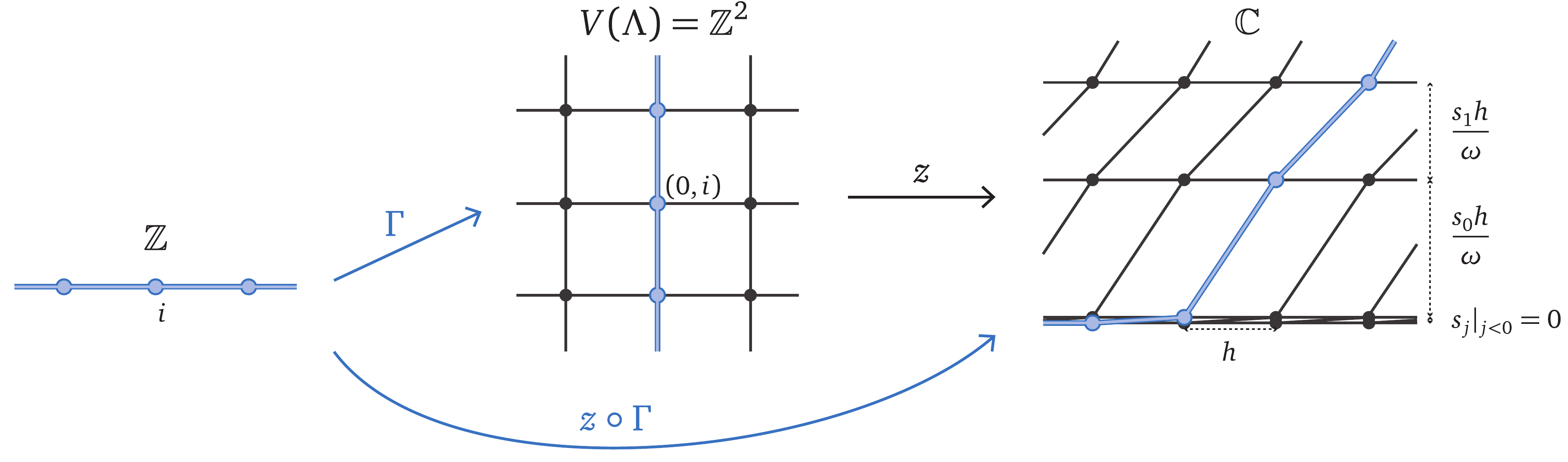}
	\caption{\label{fig:StretchingPath}The stretching path \(z\circ\Gamma\colon\ZZ\rightarrow\CC\) shares similarity in appearance with the path \(\gamma_\di\colon\RR\rightarrow\CC\) in \figref{fig:SmoothStretching}. The path \(z\circ\Gamma\) is mapped from a vertical path \(\Gamma\colon\ZZ\hookrightarrow\ZZ^2\) by the vertex positions \(z\) of a parallelogram graph. Vanishing vertical gap parameter \(s\) makes segments of \(z\circ\Gamma\) parallel to the real axis.}
\end{figure}

Continuing the setup of \secref{sec:DiscreteAnalyticContinuationOfPlaneWaves}, we consider a path \(\Gamma\colon\ZZ\hookrightarrow V(\Lambda)\), different from horizontal lines \(L_\idx{j}\), so that plane waves are attenuated along \(\Gamma\). We simply let \(\Gamma\) be the path that traverses vertically in \(\Lambda\):
\begin{align}
	\label{eq:StretchingPath}
	\Gamma(\idx{i})\coloneqq
		(0,\idx{i}).
\end{align}
In addition, we let
\begin{align*}
	s_{\idx{j}} \equiv 0,\quad\idx{j}<0
\end{align*}
for Eq.~\eqref{eq:VertexPositions}.
 Then the realization \(z\circ\Gamma\) of the path in \(\CC\) replicates a curve similar to \eqref{eq:ContinuousCurve} as shown in \figref{fig:StretchingPath}. In particular \(\re(z\circ\Gamma(\idx{i})) = \dx\idx{i}\) for all \(\idx{i}\in\ZZ\), and
\begin{align*}
	\im(z\circ\Gamma(\idx{i})) = 
	\begin{cases}
		0,&\idx{i}\leq 0\\
		\sum_{\idx{j}=0}^{\idx{i}-1}\frac{s_\idx{j}\dx}{\omega},&\idx{i}>0.
	\end{cases}
\end{align*}
Now, by \thmref{thm:DecayRate} and \thmref{thm:DecayRateLessThanOne} each holomorphic normal mode \(f_k\) of \((-2+\tau^{-1}+\tau)/\dx^2\) that corresponds to a right (\resp~left) traveling wave \(W_k = f_k\circ L_0\) will decay as we move up (\resp~down) in \(\Lambda\). Moreover, since \(s_\idx{j}\vert_{\idx{j}<0}\equiv 0\), by \lemref{lem:ZeroGap} we have \(f_k\circ L_0 = f_k\circ\Gamma\) for \(\idx{i}\leq 0\). Hence, whenever \((f_k\circ\Gamma)\vert_{\idx{i}\leq 0}\) is a traveling wave, we must have \((f_k\circ\Gamma)\vert_{\idx{i}>0}\) be a wave attenuated along its traveling direction.  Note that \(f_k\circ\Gamma\) has the exact properties we wish for a solution of a set of PML equations. 

\subsection{Discrete Complex Coordinate Stretching}
\label{sec:DiscreteComplexCoordinateStretching}

With the tools with have so far, we are ready to complexify the discrete wave equation \eqref{eq:DiscreteWaveEquationFourier} into \eqref{eq:ComplexifiedDiscreteWave}.  In saying that, we return to the \(d\)-dimensional setup. First, for each dimensional component \(\di\) of the original domain \(\ZZ^d\) we consider a discrete complex domain \((\Lambda,z^{(\di)})\) as described in \secref{sec:DiscreteAnalyticContinuationOfPlaneWaves}. More precisely, \(V(\Lambda) = \ZZ^2\) and 
\begin{align*}
	z^{(\di)}_{\idx{i,j}} = (\idx{i}+\idx{j})\dx + \iu\sum_{\tilde{\idx{j}}=0}^{\idx{j}-1}\frac{\sigma_\di(\tilde{\idx{j}})\dx}{\omega},
\end{align*}
where \(\sigma_\di(\idx{j})\geq 0, \sigma_{\di}(\idx{j})\vert_{\idx{j}<0} = 0\), is our PML damping coefficients.
Then, we extend the domain of \eqref{eq:DiscreteWaveEquationFourier} and formulate
\begin{align}
	\label{eq:ComplexifiedDiscreteWaveEquationFourier}
	-\omega^2 f = \sum_{\di=1}^{d}\frac{1}{\dx^2}\left(-2+\tau_\di^{-1}+\tau_\di\right)f,\quad f\colon (V(\Lambda))^d\rightarrow\CC,
\end{align}
where \((\tau_\di^{\pm 1}f)\left((\idx{i}_1,\idx{j}_1),\ldots,(\idx{i}_\di,\idx{i}_\di),\ldots(\idx{i}_d,\idx{j}_d)\right) = f\left((\idx{i}_1,\idx{j}_1),\ldots,(\idx{i}_\di\pm 1,\idx{i}_\di),\ldots(\idx{i}_d,\idx{j}_d)\right)\) are defined for \(V(\Lambda)\)-supported functions as described in \secref{sec:DiscreteAnalyticContinuationOfPlaneWaves}.  In addition to \eqref{eq:ComplexifiedDiscreteWaveEquationFourier}, we impose componentwise homolomorphicity: for each fixed \((\idx{i}_1,\idx{j}_1),\ldots, (\idx{i}_{\di-1},\idx{j}_{\di-1}),(\idx{i}_{\di+1},\idx{j}_{\di+1}),\ldots,(\idx{i}_d,\idx{j}_d)\),
\begin{align}
	(\idx{i}_\di,\idx{j}_\di)\mapsto f\left((\idx{i}_1,\idx{j}_1),\ldots,(\idx{i}_\di,\idx{i}_\di),\ldots(\idx{i}_d,\idx{j}_d)\right)\text{ is discrete holomorphic with respect to \(z^{(\di)}\)}.
\end{align}
Now, using the stretching path \eqref{eq:StretchingPath}, we redefine the primary variable \(\hat{U}(\omega,\bi)\), \(\bi\in\ZZ^d\), as the restriction of \(f\) along \(\Gamma\):
\begin{align}
	\label{eq:FRestrictedOnGamma}
	\hat{U}(\omega,\idx{i}_1,\ldots,\idx{i}_d)\coloneqq f(\Gamma(\idx{i}_1),\ldots,\Gamma(\idx{i}_d)).
\end{align}

\begin{remark}
	According to the theory developed in \textup{\secref{sec:DiscreteAnalyticContinuationOfPlaneWaves}}~and~\textup{\ref{sec:StretchingPath}}, \emph{by construction}
	such a function \(\hat{U}\) (\textup{Eq.~\eqref{eq:FRestrictedOnGamma}}) arising from a solution \(f\) of Eq.~\eqref{eq:ComplexifiedDiscreteWaveEquationFourier} is dictated to be an attenuated wave on \(\bigcup_{\di=1}^d\{\idx{i}_\di\geq 0\}\) whenever \(\hat{U}\) is a \revise{traveling} wave \(\prod_{\di=1}^d W_{k_\di}(\idx{i}_\di)\) on \(\bigcap_{\di=1}^d\{\idx{i}_\di<0\}\).
\end{remark}

\begin{figure}
	\centering
	\includegraphics[width=0.65\textwidth]{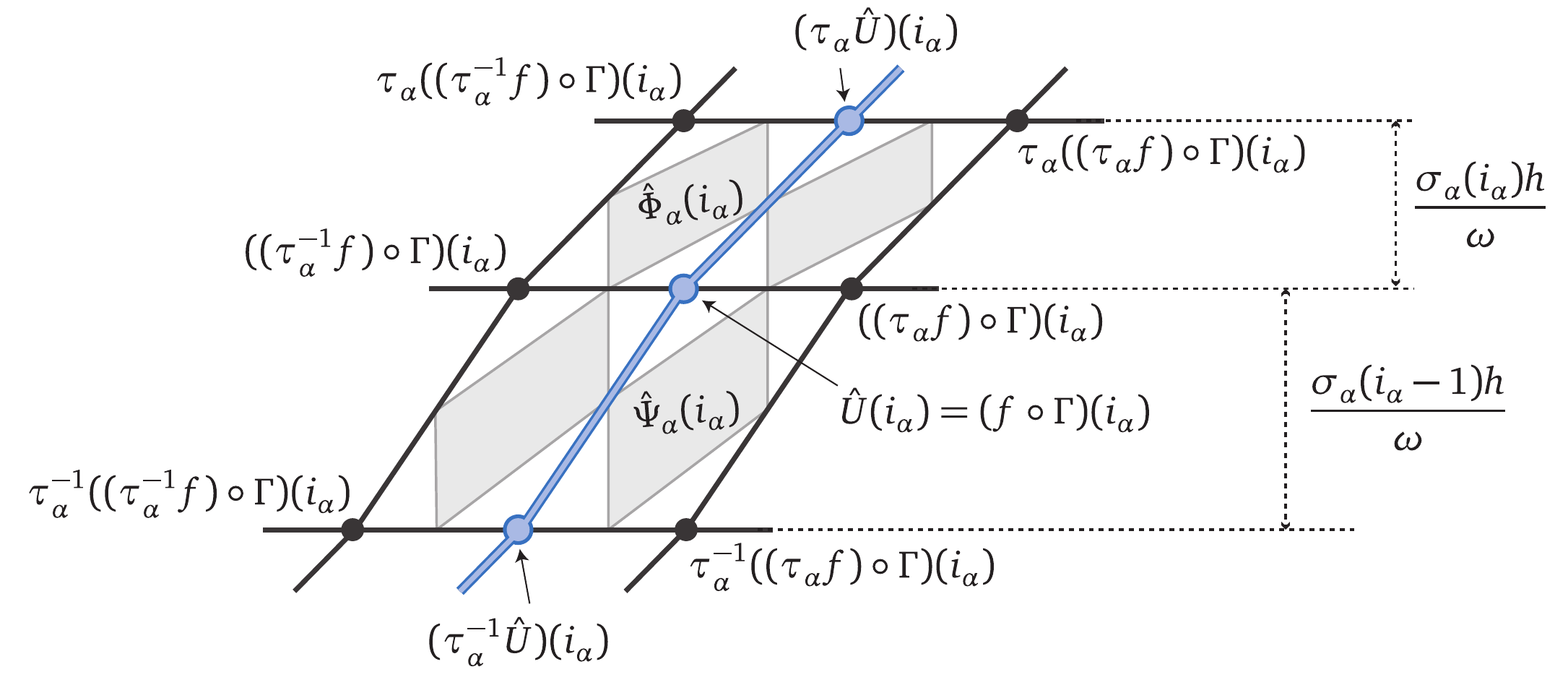}
	\caption{\label{fig:Stencils}An illustration showing at which vertex or quadrilateral of \(\Lambda\) each quantity lives.  The figure only shows a neighborhood of the image of \(\Gamma_\di\) in an \(\di\)-component slice of \((V(\Lambda))^d\). The combinatorial \(\Lambda\) is visualized using the vertex positions \(z_\di\).  Since only the \(\di\)-component is concerned, we use, \eg, \(\tau_\di((\tau_\di^{-1}\circ f)\circ\Gamma)(\idx{i}_\di)\) as an abbreviation for \((\tau_\di^{-1}\circ f)(\Gamma(\idx{i}_1),\ldots,\Gamma(\idx{i}_\di+1),\ldots,\Gamma(\idx{i}_d))\).} 
\end{figure}

Finally,
to arrive at the PML equations we need to rewrite Eq.~\eqref{eq:ComplexifiedDiscreteWaveEquationFourier} as an equation for \(\hat{U}\), and arrange it in such a form that we can transform back to the time domain.
Similar to the continuous counterpart Eq.~\eqref{eq:VHatDefinition}, we consider the \(\nicefrac{1}{(\iu\omega)}\)-scaled discrete complex derivative (in the sense of \defref{def:ComplexDifferentiable})
\begin{align}
	\label{eq:DiscretePMLVelocity} \frac{1}{\iu\omega}\left(D_{z^{(\di)}}f\right)\big((\idx{i}_1,\idx{j}_1),\ldots,Q_\idx{i_\di+\half,j_\di+\half},\ldots,(\idx{i}_d,\idx{j}_d)\big)
\end{align}
at each quadrilateral \(Q_\idx{i_\di+\half,j_\di+\half}\) incident to the image of \(\Gamma\) in each dimension.  Explicitly, using \defref{def:ComplexDifferentiable} we let the quantity \eqref{eq:DiscretePMLVelocity} at the quadrilaterals to the \emph{upper-left} of \(\Gamma\) (\figref{fig:Stencils}) be denoted by
\begin{subequations}
	\label{eq:PhiDefinition}
\begin{align}
	\label{eq:PhiDefinitionA}
    \nonumber
	&\hat\Phi_\di(\omega,\idx{i}_1,\ldots,\idx{i}_d)\\
	&\coloneqq
	\frac{1}{\iu\omega}\cdot\frac{1}{2\dx + \iu\frac{\sigma_\di(\idx{i}_\di)}{\omega}\dx}
	\bigg(
	f(\Gamma(\idx{i}_1),\ldots,\Gamma(\idx{i}_\di+1),\ldots,\Gamma(\idx{i}_d))
	-
	(\tau_\di^{-1} f)(\Gamma(\idx{i}_1),\ldots,\Gamma(\idx{i}_\di),\ldots,\Gamma(\idx{i}_d))
	\bigg)\\
	\label{eq:PhiDefinitionB}
	&\overset{\eqref{eq:DiscreteCauchyRiemann}}{=}
	\frac{1}{\iu\omega}\cdot
	\frac{1}{\iu\frac{\sigma_\di(\idx{i}_\di)}{\omega}\dx}
	\bigg(
	(\tau_\di^{-1}f)(\Gamma(\idx{i}_1),\ldots,\Gamma(\idx{i}_\di+1),\ldots,\Gamma(\idx{i}_d))
	-
	f(\Gamma(\idx{i}_1),\ldots,\Gamma(\idx{i}_\di),\ldots,\Gamma(\idx{i}_d))
	\bigg);
\end{align}
\end{subequations}
and we let \eqref{eq:DiscretePMLVelocity} at the quadrilaterals to the \emph{lower-right} of \(\Gamma\) be denoted by
\begin{subequations}
	\label{eq:PsiDefinition}
\begin{align}
	\label{eq:PsiDefinitionA}
    \nonumber
	&\hat\Psi_\di(\omega,\idx{i}_1,\ldots,\idx{i}_d)\\
	&\coloneqq
	\frac{1}{\iu\omega}\cdot\frac{1}{2\dx + \iu\frac{\sigma_\di(\idx{i}_\di-1)}{\omega}\dx}
	\bigg(
	(\tau_\di f)(\Gamma(\idx{i}_1),\ldots,\Gamma(\idx{i}_\di),\ldots,\Gamma(\idx{i}_d))
	-
	f(\Gamma(\idx{i}_1),\ldots,\Gamma(\idx{i}_\di-1),\ldots,\Gamma(\idx{i}_d))
	\bigg)\\
	\label{eq:PsiDefinitionB}
	&\overset{\eqref{eq:DiscreteCauchyRiemann}}{=}
	\frac{1}{\iu\omega}\cdot
	\frac{1}{\iu\frac{\sigma_\di(\idx{i}_\di-1)}{\omega}\dx}
	\bigg(
	f(\Gamma(\idx{i}_1),\ldots,\Gamma(\idx{i}_\di),\ldots,\Gamma(\idx{i}_d))
	-
	(\tau_\di f)(\Gamma(\idx{i}_1),\ldots,\Gamma(\idx{i}_\di-1),\ldots,\Gamma(\idx{i}_d))
	\bigg).
\end{align}
\end{subequations}
By substituting \eqref{eq:FRestrictedOnGamma} and noting that
\begin{align*}
	f(\Gamma(\idx{i}_1),\ldots,\Gamma(\idx{i}_\di\pm 1),\ldots,\Gamma(\idx{i}_d)) = (\tau_\di^{\pm 1}\hat U)(\omega,\idx{i}_1,\ldots,\idx{i}_d),
\end{align*}
we can simplify \eqref{eq:PhiDefinitionA} as
\begin{align}
	\label{eq:AlmostPhiEquation}
	(\tau_\di^{-1}f)(\Gamma(\idx{i}_1),\ldots,\Gamma(\idx{i}_d)) = (\tau_\di\hat{U})(\omega,\idx{i}_1,\ldots,\idx{i}_d) - \left(2\iu\omega\dx - \sigma_\di(\idx{i}_\di)\dx\right)\hat\Phi_\di(\omega,\idx{i}_1,\ldots,\idx{i}_d).
\end{align}
On the other hand, by applying \(\tau_\di^{-1}\) to \(\hat{\Phi}_\di\) with the formula \eqref{eq:PhiDefinitionB}, \ie~by replacing \(\idx{i}_\di\) with \(\idx{i}_\di-1\) in \eqref{eq:PhiDefinitionB}, we also have
\begin{align}
	\label{eq:FLeftFormula}
	(\tau_\di^{-1}f)(\Gamma(\idx{i}_1),\ldots,\Gamma(\idx{i}_d)) = (\tau_\di^{-1}\hat U)(\omega,\idx{i}_1,\ldots,\idx{i}_d) - \dx(\tau_\di^{-1}\sigma_\di)(\idx{i}_\di)(\tau^{-1}\hat \Phi_\di)(\omega,\idx{i}_1,\ldots,\idx{i}_d).
\end{align}
Then by replacing the left-hand side of \eqref{eq:AlmostPhiEquation} with  \eqref{eq:FLeftFormula}, we obtain
\begin{align*}
	\tau_\di^{-1}\hat{U} - \dx(\tau_\di^{-1}\sigma_\di)(\tau^{-1}_\di\hat\Phi_\di) = (\tau_\di\hat U) - (2\iu\omega\dx - \sigma_\di\dx)\hat\Phi_\di,
\end{align*}
which is rearranged as an equation for \(\hat{\Phi}_\di\):
\begin{align}
	\label{eq:PhiEquationFourier}
	-\iu\omega\hat\Phi_\di = -\frac{1}{2}\left((\tau_\di^{-1}\sigma_\di)(\tau_\di^{-1}\hat\Phi_\di) + \sigma_\di\hat\Phi_\di\right) - \frac{1}{2\dx}\left(\tau_\di\hat U - \tau_\di^{-1}\hat U\right).
\end{align}
By a similar procedure but applied to \eqref{eq:PsiDefinition}, we get the formula analogous to \eqref{eq:FLeftFormula}:
\begin{align}
	\label{eq:FRightFormula}
	(\tau_\di f)(\Gamma(\idx{i}_1),\ldots,\Gamma(\idx{i}_d)) = (\tau_\di\hat U)(\omega,\idx{i}_1,\ldots,\idx{i}_d) + \dx\sigma_\di(\idx{i}_\di)(\tau\hat{\Psi}_\di)(\omega,\idx{i}_1,\ldots,\idx{i}_d),
\end{align}
and an equation for \(\hat{\Psi}_\di\) (analogous to \eqref{eq:PhiEquationFourier}):
\begin{align}
	\label{eq:PsiEquationFourier}
	-\iu\omega\hat\Psi_\di = -\frac{1}{2}\left((\tau_\di^{-1}\sigma_\di)\hat\Psi_\di + \sigma_\di(\tau_\di\hat\Psi_\di)\right) - \frac{1}{2h}\left(\tau_\di\hat U - \tau_\di^{-1}\hat U\right).
\end{align}
Now, we evaluate Eq.~\eqref{eq:ComplexifiedDiscreteWaveEquationFourier} on \((\Gamma(\idx{i}_1),\ldots,\Gamma(\idx{i}_d))\), where we may substitute not only \(f(\Gamma(\idx{i}_1),\ldots,\Gamma(\idx{i}_d))\) by \(\hat{U}\), but also \((\tau_\di^{\pm 1}f)(\Gamma(\idx{i}_1),\ldots,\Gamma(\idx{i}_d))\) with \eqref{eq:FLeftFormula} and \eqref{eq:FRightFormula}.  This turns Eq.~\eqref{eq:ComplexifiedDiscreteWaveEquationFourier} into
\begin{align}
	\label{eq:UEquationFourier}
	-\omega^2\hat U = 
	\sum_{\di=1}^d \frac{1}{\dx^2}\left(-2\hat U + \tau_\di^{-1}\hat U + \tau_\di \hat U\right) + \sum_{\di=1}^d\frac{1}{\dx}\left(\sigma_\di(\tau_\di \hat\Psi_\di) - (\tau_\di^{-1}\sigma_\di)(\tau_\di^{-1}\hat\Phi_\di)\right).
\end{align}
By applying the inverse Fourier transform to \eqref{eq:UEquationFourier}, \eqref{eq:PhiEquationFourier} and \eqref{eq:PsiEquationFourier}, we complete the derivation of the discrete PML equations~\eqref{eq:DiscretePML}.
\thmref{thm:Reflectionless} also follows by the remark at Eq.~\eqref{eq:FRestrictedOnGamma}, \thmref{thm:DecayRate} and \thmref{thm:DecayRateLessThanOne}.

\section{Conclusion}
\label{sec:Conclusion}

This paper presents a derivation of a discrete PML that perfectly matches the multi-dimensional discrete wave equation. By construction, the discrete PML produces no numerical reflection.  Similar to most previous PMLs, it takes a finite difference form with a local compact stencil, which is ideal as an efficient boundary treatment.  After the domain is truncated, we observe numerical stability and exponentially small residual waves which can be practically maintained below machine zero.  

The arrival at such a discrete PML perhaps sheds some light on numerical approaches on a larger scope.
In the past quarter century, PML has been discretized through the methods of finite difference, finite element, discontinuous Galerkin, \etc\ These discretization schemes are deeply rooted in the approximation theory. Numerical reflections come with no surprise when a discrete PML is derived following this route.  On the other hand, if the already-discretized wave equation is treated as a discrete problem in its own right, it is possible to arrive at a reflectionless PML that bears similar properties promised by the continuous PML.  On the way of the derivation, one faces the mathematical questions such as ``What is the discrete analog of analytic continuation?'' and ``Can a finite difference operator be extended seamlessly to a discrete complex domain?'' Often one learns deeper insights into the fundamental structure of the continuous problem just by answering these discrete questions.

The PML presented in this work is for one specific discretized scalar wave equation.  
\revise{In other words, it is not a general theory.
Each discrete wave equation requires a new investigation in the perspective of discrete complex analysis and discrete differential geometry.   
A discrete wave equation based on a higher order scheme may require a different discrete Cauchy--Riemann relation. 
Designing PMLs that are coupled with a real grid stretching \cite{Kreiss:2016:ASG} may require a different analytic continuation of the wave equation.
Another notable property of the current discrete wave equation is that it allows one to investigate the PML dimension by dimension, each of which fortunately admits a straightforward discrete complex extension.
Such a property may not be true for other discretized wave equations. 
A discrete wave equation formulated on an unstructured mesh also requires new coordinate-free theories for discussing any notion of discrete analytic continuation.
Such a question may be closely related to the study of high-dimensional complex manifolds.
The current work also relies on taking Fourier transform in time which separate the temporal coordinate from the spatial ones.  Such a treatment may need to be revised for the discrete Maxwell's equations based on Yee's scheme (FTDT methods), which should be viewed as a discretization of electro\-magnetism in relativistic form \cite{stern:2015:GCE}.

The stability of the presented discrete PML is only shown empirically.  Since the discrete PML is constructed in analog to the one in the continuous setup, it is expected that the stability analysis in the continuous theory admits a discrete analog that proves the stability for the discrete PML.  Developing such a discrete theory would also be an important research direction.

With these wide open questions, it is exciting to expect a new generation of investigations into designing reflectionless discrete PMLs for a variety of discrete wave equations. 
}

 \appendix

 \section{Proof of \thmref{thm:DecayRate}}
 \label{app:ProofOfThmDecayRate}
The result of \thmref{thm:DecayRate} only requires showing Eq.~\eqref{eq:DecayRate} for \(\idx{j}=1\), which is a recursive relation that completes the rest of the proof.  Specifically, we want to show that if \(f\in\CC^{V(\Lambda)}\) is a holomorphic eigenfunction of \((-2+\tau^{-1}+\tau)/\dx^2\), and \(f\circ L_0 = W_k\), then \(f\circ L_1 = \rho(s_0,k,\omega)\tau W_k\).

First, let us speculate the eigenspace \(\mathcal{Y}_\lambda\) of the operator \((-2+\tau^{-1}+\tau)/\dx^2\vert_{\CC^\ZZ}\) for all its eigenvalues \(\lambda\).  In addition to \(W_k(\idx{i}) = e^{\iu k\dx\idx{i}}\) defined in \eqref{eq:DiscreteMode}, we let \(A\colon\ZZ\rightarrow\CC\) denote the ``linear'' function
\begin{align*}
	A(\idx{i}) \coloneqq \idx{i}.
\end{align*}
From the standard finite difference theory, we learn that the full spectrum and the associated eigenspaces of the operator \((-2+\tau^{-1}+\tau)/\dx^2\vert_{\CC^\ZZ}\) are given by
\begin{subequations}
	\label{eq:FullSpectrum}
	\begin{alignat}{2}[left={\mathcal{Y}_\lambda = \empheqlbrace}]
		\label{eq:FullSpectrumA}
		&\textstyle\spanset\left\{W_k,W_{-k}\right\},
		&&\quad\textstyle \lambda = -\frac{4}{\dx^2}\sin^2\left(\frac{k\dx}{2}\right),\quad k\in\CC\setminus\left\{\nicefrac{n\pi}{\dx}\,\vert\, n\in\ZZ\right\},\\
		\label{eq:FullSpectrumB}
		&\textstyle\spanset\left\{W_0, A\right\},
		&&\quad\textstyle \lambda = 0,\\
		\label{eq:FullSpectrumC}
		&\textstyle\spanset\left\{W_{\pi/\dx}, AW_{\pi/\dx}\right\}, 
		&&\quad\textstyle\lambda = -\frac{4}{\dx^2}.
	\end{alignat}
\end{subequations}
Note that by the periodicity of the complex sine function, the spectrum of the case  \eqref{eq:FullSpectrumA} can be bijectively parameterized by
\begin{align*}
	\revise{
    k\in \cK\setminus((-\nicefrac{\pi}{\dx},0)\cup\{\nicefrac{\pi}{\dx}\}).
    }
\end{align*}
Now, knowing that the eigenfunction \(f\in\CC^{V(\Lambda)}\) we are looking for already satisfies \(f\circ L_0 = W_k\) with  \revise{\(k\in\cK\)}, we have the candidate for \(f\circ L_1\) limited to \eqref{eq:FullSpectrum} with \(\lambda = -\nicefrac{4}{\dx^2}\sin^2(\nicefrac{k\dx}{2})\) with \revise{\(k\in\cK\) and } \(0\leq k\leq \revise{\nicefrac{\pi}{\dx}}\).  Keep in mind that the cases \(k=0, \revise{\nicefrac{\pi}{\dx}}\) invoke \eqref{eq:FullSpectrumB} and \eqref{eq:FullSpectrumC} respectively.

Next, we apply the holomorphicity condition \eqref{eq:DiscreteCauchyRiemann} to the quadrilaterals sandwiched by \(L_0\) and \(L_1\).  Observe that \(\tau z - z=\dx\), the discrete Cauchy--Riemann equation is \(\tau\)-translationally invariant.  This allows \eqref{eq:DiscreteCauchyRiemann} to take a more elagent form
\begin{align*}
	\frac{\tau f\circ L_1 - f\circ L_0}{2\dx+\iu\frac{s_0}{\omega}\dx} = \frac{f\circ L_1 - \tau f\circ L_0}{\iu\frac{s_0}{\omega}\dx},
\end{align*}
which can be further rearranged into
\begin{align}
	\label{eq:Recursion}
	\left(2 + \iu\frac{s_0}{\omega}(1-\tau)\right)(f\circ L_1) = \left(2 + \iu\frac{s_0}{\omega}(1-\tau^{-1})\right)\tau (\underbrace{f\circ L_0}_{=\,W_k}).
\end{align}

Now, we discuss the cases  \revise{\(k\in\cK\setminus\{0,\nicefrac{\pi}{\dx}\}\)}, \(k=0\) and \(k = \nicefrac{\pi}{\dx}\) separately.  The following facts are useful down the road:
\begin{align}
	\tau W_{\pm k} = e^{\pm\iu k \dx}W_k,\quad \tau A = A+1.
\end{align}

\paragraph{Case  \revise{\(k\in\cK\setminus\{0,\nicefrac{\pi}{\dx}\}\)}}
By \eqref{eq:FullSpectrumA} \(f\circ L_1\) takes a general form
\begin{align}
	\label{eq:Ansaltz1}
	f\circ L_1 = c^+W_k + c^- W_{-k}
\end{align}
for \(c^+,c^-\in\CC\). Substituting \eqref{eq:Ansaltz1} to \eqref{eq:Recursion} and by linear independence of \(W_k\) from \(W_{-k}\) we must have \(c^-=0\) or 
\begin{align}
	\label{eq:NonzeroC-}
	\left(2 + \iu\frac{s_0}{\omega}(1-\tau)\right)W_{-k} = 0\quad\implies\quad 2 + \iu\frac{s_0}{\omega}(1-e^{-\iu k\dx}) = 0.
\end{align}
The latter case \eqref{eq:NonzeroC-} never happens for \(s_0,\omega\in\RR\) and  \revise{\(k\in\cK\setminus\{0,\nicefrac{\pi}{\dx}\}\)}. Therefore \(c^- = 0\).  Now with both \(f\circ L_1\) and \(f\circ L_0\) multiples of \(W_k\), which is an eigenfunction of \(\tau\) with eigenvalue \(e^{\iu k\dx}\), Eq.~\eqref{eq:Recursion} immediately gives us
\begin{align*}
	f\circ L_1 = \frac{2+\iu\frac{s_0}{\omega}(1-e^{-\iu k\dx})}{2+\iu\frac{s_0}{\omega}(1-e^{\iu k\dx})}\tau W_k = \rho(s_0,k,\omega)\tau W_k.
\end{align*}
\qed

\paragraph{Case \(k=0\)}
By \eqref{eq:FullSpectrumB} we have 
\(
	f\circ L_1 = c_0 + c_1 A
\)
for some \(c_0,c_1\in\CC\), which turns \eqref{eq:Recursion} into
\begin{align*}
	2c_0 + 2c_1 A - c_1\iu\frac{s_0}{\omega} = 2.
\end{align*}
This implies \(c_0 = 1\) and \(c_1 = 0\).  Therefore, \(f\circ L_1 = 1 = \rho(s_0,0,\omega)\tau W_0\).\qed

\paragraph{Case \(k=\nicefrac{\pi}{\dx}\)}
By \eqref{eq:FullSpectrumB}, we express
\(
	f\circ L_1 = c_0W_{\pi/\dx} + c_1 AW_{\pi/\dx},
\)
\(c_0,c_1\in\CC\).  Note that \(\tau W_{\pi/\dx} = -W_{\pi/\dx}\) and \(\tau(AW_{\pi/\dx}) = (\tau A)(\tau W_{\pi/\dx}) = -AW_{\pi/\dx} - W_{\pi/\dx}\).
Hence \eqref{eq:Recursion} becomes
\begin{align*}
	c_0\left(2+2\iu\frac{s_0}{\omega}\right)W_{\pi/\dx} + c_1\left(2+2\iu\frac{s_0}{\omega}\right)AW_{\pi/\dx} + c_1 \iu \frac{s_0}{\omega}W_{\pi/\dx} = -\left(2+2\iu\frac{s_0}{\omega}\right)W_{\pi/\dx}.
\end{align*}
Since \(AW_{\pi/\dx}\) is linearly independent from \(W_{\pi/\dx}\), and \((2+2\iu\nicefrac{s_0}{\omega})\neq 0\), we must have \(c_1 = 0\), and consequently \(c_0 = -1\).  Therefore, \(f\circ L_1 = -W_{\pi/\dx} = \rho(s_0,\nicefrac{\pi}{\dx},\omega)\tau W_{\pi/\dx}\).\qed

\revise{
\section{Proof of \thmref{thm:DecayRateLessThanOne}}
\label{app:ProofOfThmDecayRateLessThanOne}
Here we prove that for every \(s>0\), \(\omega\in\RR\), \(\omega\neq 0\) and \(k\in\cK\) with \(0<\sgn(\omega)\re(k)<\nicefrac{\pi}{\dx}\), the expression 
\begin{align}
    \label{eq:AppRhoDefinition}
    \rho(s,k,\omega)=\frac{2+\iu\frac{s}{\omega}(1-e^{-\iu k\dx})}{2+\iu\frac{s}{\omega}(1-e^{\iu k \dx})}
\end{align}
satisfies 
\begin{align*}
    |\rho(s,k,\omega)|<1\quad\text{and}\quad \frac{|\rho(s,k,\omega)|}{\left|\rho\left(s,-\conj{k},\omega\right)\right|}<1.
\end{align*}
Note that \eqref{eq:AppRhoDefinition} has the symmetry \(\rho(s,k,-\omega) = \conj{\rho(s,-\conj{k},\omega)}\), which allows us to consider only the case when \(\omega>0\) and \(k\in\cK\), \(\re(k)\in(0,\nicefrac{\pi}{\dx})\).  

For notational simplicity, define \(\zeta\coloneqq e^{\iu k\dx}\) and \(a\coloneqq \nicefrac{2\omega}{s}\), which satisfy
\begin{align}
    \label{eq:AppZetaCondition}
    |\zeta|>1,\quad \im(\zeta)>0,\quad\text{and}\quad a>0.
\end{align}
  Then we can write \eqref{eq:AppRhoDefinition} as
\begin{align*}
    \rho(s,k,\omega) = \frac{1-\iu a-\nicefrac{1}{\zeta}}{1-\iu a - \zeta},\quad
    \rho\left(s,-\conj k,\omega\right)
    = \frac{1-\iu a-\nicefrac{1}{\conj\zeta}}{1-\iu a-\conj\zeta}.
\end{align*}
We show \(|\rho(s,k,\omega)|<1\) by checking \(|1-\iu a-\zeta|^2>|1-\iu a-\nicefrac{1}{\zeta}|^2\).  By expansion, this is equivalent to verifying
\begin{align}
    \label{eq:AppInequalityClaim}
    |\zeta|^2-2\re(\zeta)+2a\im(\zeta)>\frac{1}{|\zeta|^2}-2\re\left(\tfrac{1}{\zeta}\right)+2a\im\left(\tfrac{1}{\zeta}\right).
\end{align}
To show \eqref{eq:AppInequalityClaim}, first note that whenever \(|\zeta|>1\) we have
\begin{align*}
    |\zeta|^2-2\re(\zeta)>\frac{1}{|\zeta|^2}-2\re\left(\tfrac{1}{\zeta}\right),
\end{align*}
which follows immediately from \(\re(\nicefrac{1}{\zeta}) = \nicefrac{\re(\zeta)}{|\zeta|^2}\) and the expansion of \(|\zeta|^2|1-\zeta|^2>|1-\zeta|^2\).  The claimed inequality \eqref{eq:AppInequalityClaim} then follows since  \(2a\im(\zeta)>0\) and \(2a\im(\nicefrac{1}{\zeta})<0\) as implied by \eqref{eq:AppZetaCondition}.  Consequently, \(|\rho(s,k,\omega)|<1\).

Finally, let us prove  \(\nicefrac{|\rho(s,k,\omega)|}{|\rho(s,-\conj{k},\omega)|}<1\).  That is, we want to show
\begin{align}
    \label{eq:AppInequalityClaim2}
    \frac{|1-\iu a - \nicefrac{1}{\zeta}|}{|1-\iu a - \zeta|}<\frac{\left|1-\iu a - \nicefrac{1}{\conj\zeta}\right|}{\left|1-\iu a - \conj\zeta\right|}.
\end{align}
Observe that \((1-\iu a)\) lies in the lower half plane and \(\zeta\) lies in the upper half plane.  Therefore, the distance from \((1-\iu a)\) to \(\zeta\) is greater than that to \(\conj\zeta\).  In other words, \(|1-\iu a - \zeta|>|1-\iu a - \conj{\zeta}|\).  Similarly, \(\nicefrac{1}{\zeta}\) lies in the lower half plane, so \(|1-\iu a - \nicefrac{1}{\conj\zeta}|>|1-\iu a - \nicefrac{1}{\zeta}|\).
The inequality \eqref{eq:AppInequalityClaim2} then follows.
\qed
}

\section*{Acknowledgments}
This work has been supported by DFG Collaborative Research Center TRR 109 ``Discretization in Geometry and Dynamics.''  Additional support was provided by SideFX software.  I especially thank Prof.~Olof Runborg at KTH, Stockholm, for setting forth the PML problem and many useful discussions.



\bibliographystyle{model1-num-names}
\bibliography{DiscretePML.bib}

\end{document}